\documentclass[11pt]{amsart}

\usepackage[margin=1.12in]{geometry}
\usepackage{amsmath,amssymb,amsthm,mathtools}
\usepackage{graphicx}
\usepackage{enumitem}
\usepackage{microtype}
\usepackage[hidelinks]{hyperref}
\usepackage{cleveref}
\usepackage{tikz-cd}
\usetikzlibrary{arrows.meta,positioning,calc,decorations.pathmorphing}
\usepackage{xcolor}
\usepackage{booktabs}
\usepackage{array}

\definecolor{lineblue}{RGB}{20,82,180}
\definecolor{linegreen}{RGB}{0,132,92}
\definecolor{quartic}{RGB}{185,55,115}
\definecolor{linered}{RGB}{205,45,55}
\definecolor{linepink}{RGB}{225,90,155}
\definecolor{lineorange}{RGB}{225,120,25}
\definecolor{linepurple}{RGB}{115,70,170}
\definecolor{lineteal}{RGB}{0,145,155}

\newtheorem{theorem}{Theorem}[section]
\newtheorem{proposition}[theorem]{Proposition}
\newtheorem{lemma}[theorem]{Lemma}

\theoremstyle{definition}

\theoremstyle{remark}
\newtheorem{remark}[theorem]{Remark}

\newcommand{\Mod}{\operatorname{Mod}}

\newcommand{\e}{\mathrm{e}}
\newcommand{\sig}{\sigma}
\newcommand{\ch}{\chi_h}

\title[Xiao's degree-four fibration]{Xiao's Degree-Four Fibration and the Polizzi Model}
\author{Anar Akhmedov}
\thanks{A. Akhmedov is also a Visiting Scholar in the Department of Mathematics at Harvard University, One Oxford Street, Cambridge, MA 02138, USA}
\address{School of Mathematics, University of Minnesota, Minneapolis, MN 55455, USA; Visiting Scholar, Department of Mathematics, Harvard University, One Oxford Street, Cambridge, MA 02138, USA}
\email{akhmedov@umn.edu}
\email{akhmedov@math.harvard.edu}
\author{Sumeyra Sakalli}
\address{Department of Mathematics and Statistics, University of South Florida, Tampa, FL 33620, USA}
\email{sumeyras@usf.edu}
\date{October 2026}
\dedicatory{Dedicated to the memory of Professor Gang Xiao}

\hypersetup{
  pdftitle={Xiao's Degree-Four Fibration and the Polizzi Model},
  pdfauthor={Anar Akhmedov and Sumeyra Sakalli}
}


\begin{document}
\raggedbottom
\begin{abstract}
Let $S_E$ be the surface in Xiao's degree-four family identified by Polizzi
as a smooth divisor in $E(3)$.  Starting from Polizzi's branch model and
Xiao's classification of the thirteen singular fibers, we derive the local
spherical braids and their Picard--Lefschetz lifts, including the colored
$3+3$ partitions at the seven reducible fibers.  We identify the
normalization of each bisection with $E/\{\pm1\}$ and its degree-two ruling
map with the quotient by the involution induced by translation by a nonzero
two-torsion point.  We prove $\pi_1(S_E)\cong\mathbb Z^2$, compute the
elementary-divisor-$4$ Albanese kernel of a regular genus-two fiber, and show
that the six nonseparating Picard--Lefschetz transformations represent the
six cusps of $\Gamma(4)$.  We also prove that the natural product tori near
an elliptic section are nullhomologous and cannot lower $b_1$ below $2$ by
torus surgery.  As a separate application of the degree-three monodromy, we
give a twisted-double construction of an exotic
$\mathbb CP^2\#7\overline{\mathbb CP}^{\,2}$.
\end{abstract}

\maketitle

\section{Introduction}
\enlargethispage{4pt}
Catanese and Ciliberto constructed the moduli space of surfaces with
$p_g=q=1$, $K^2=3$, and Albanese genus three, realizing their canonical
models as divisors in $E(3)$ \cite{CataneseCiliberto}.  Xiao constructed the
one-dimensional subfamily carrying a rational genus-two fibration
\cite{Xiao}.  Polizzi subsequently identified this subfamily inside the
Catanese--Ciliberto model and developed its detailed double-cover and
branch-divisor geometry \cite{Polizzi}.

Akhmedov and Monden first developed Xiao's degree-three fibration as a
symplectic four-manifold building block and used it in constructions of
small exotic manifolds \cite{AkhmedovMonden}.  Akhmedov subsequently
computed its ordered branch-cover monodromy \cite{AkhmedovXiao}.  The
present paper concerns principally the degree-four Polizzi model; the final
section records a twisted-double application of the degree-three
factorization.  Independently, Szab\'o has determined marked vanishing
cycles and positive factorizations for Xiao's fibrations in degrees $3$,
$4$, and $5$ by a triangular-tiling and local-hexagon method
\cite{SzaboXiao}.  Our degree-four analysis instead uses Polizzi's ruled
surface branch divisor.  It identifies the local branch degenerations and
their Picard--Lefschetz lifts and relates them to the Albanese lattice; it
does not use Szab\'o's factorization.

Let $S_E$ denote the member corresponding to the elliptic curve $E$.
Polizzi's construction gives a smooth minimal surface of general type
\cite[Theorem~7.1]{Polizzi} with
\[
p_g(S_E)=q(S_E)=1,\qquad K_{S_E}^2=3.
\]
After the standard translation of $E(3)$, Polizzi identifies the relevant
member with the smooth divisor
\[
S_E=X_E\in|4D_0-F_0|\subset E(3),
\]
where $E(3)\to E$ is the third symmetric product of the elliptic curve $E$.
The surface carries two distinct fibrations: the genus-three Albanese
fibration $S_E\to E$ and Xiao's genus-two fibration $S_E\to\mathbb P^1$.
Xiao's classification, recalled in \cite[Proposition~6.1]{Polizzi}, gives
six irreducible one-nodal fibers of arithmetic genus two and seven reducible
fibers whose components are smooth elliptic curves meeting transversely
once.  A regular genus-two fiber meets a regular Albanese fiber in four
points \cite[Proposition~6.18]{Polizzi}.

The motivating problem is to construct an exotic
$\mathbb CP^2\#6\overline{\mathbb CP}^{\,2}$ from Polizzi's surface.  We
analyze the monodromy, fundamental group, Albanese-kernel lattice, and the
most immediate Lagrangian-torus surgery candidates.  The latter candidates
are obstructed, and no exotic manifold is obtained directly from $S_E$.

Our main results concern the fibration on $S_E$ itself.  A final,
independent section treats a twisted double of the type-$(4,3)$ Xiao
fibration.  The
topological Lefschetz hyperplane theorem applied to the ample divisor
$X_E\subset E(3)$ gives
\[
\pi_1(S_E)\cong\mathbb Z^2.
\]
If $\Sigma$ is a regular genus-two fiber and
$h=\alpha|_\Sigma:\Sigma\to E$, then
\[
K_0=\ker\bigl(h_*:H_1(\Sigma;\mathbb Z)\to H_1(E;\mathbb Z)\bigr)
\]
is a primitive rank-two lattice with alternating form $4J$.  The six
nonseparating Picard--Lefschetz transformations represent the six distinct
cusps of $\Gamma(4)$.  Consequently no two nonseparating critical points
form a matching pair; see \cref{thm:pi1-polizzi,thm:polizzi-vanishing-lattice}.

Section~\ref{sec:F2} derives the local twist types from Polizzi's branch
divisor together with Xiao's singular-fiber classification as recalled in
\cite[Proposition~6.1]{Polizzi}.  The later discussion of Lagrangian product
tori proves that the natural nullhomologous tori near an elliptic section
cannot kill the two Albanese generators by torus surgery; see
\cref{thm:nullhomologous-surgery}.

The $E(3)$ framework is due to Catanese--Ciliberto, the
genus-two family and its $6+7$ fiber classification are due to Xiao, and
Polizzi identifies the family inside that framework and supplies the
explicit reducible components, intersection data, and $\mathbb F_2$ branch
model.  Akhmedov and Monden's earlier work \cite{AkhmedovMonden} provided a
four-manifold treatment of Xiao's fibration; the present paper builds
directly on that perspective rather than introducing Xiao's fibration into
the subject for the first time.  The contributions of this
paper are a branch-theoretic description
of the local type-$(6,7)$ monodromy, an explicit description of the three
normalized bisection covers, the topological computation of
$\pi_1(S_E)$, the elementary-divisor-$4$
Albanese-kernel and six-cusp theorem, and the homological obstruction to the
local nullhomologous-torus surgery strategy.  We also prove simple
connectivity of a particular twisted double of the type-$(4,3)$ block.  No
exotic six-blowup manifold is claimed in this paper; the direct Polizzi and
lantern constructions remain conditional.

\section{Genus-two monodromy and numerical formulas}

Let $\Sigma_2$ be a closed oriented surface of genus two.  Write $t_c$ for the positive Dehn twist about a simple closed curve $c\subset\Sigma_2$.  A genus-two Lefschetz fibration over $S^2$ determines a positive factorization
\[
t_{c_1}\cdots t_{c_N}=1\qquad\text{in }\Mod(\Sigma_2).
\]
We call its \emph{type} $(n,s)$ if exactly $n$ vanishing cycles are nonseparating and $s$ are separating.

\begin{proposition}\label{prop:numerics}
For a genus-two Lefschetz fibration $f:X\to S^2$ of type $(n,s)$,
\[
\e(X)=n+s-4,
\qquad
\sig(X)=-\frac35n-\frac15s.
\]
Consequently
\[
\ch(X)=\frac{\e(X)+\sig(X)}4,
\qquad
c_1^2(X)=2\e(X)+3\sig(X).
\]
\end{proposition}

\begin{proof}
The Euler characteristic formula is the standard Lefschetz-fibration formula
$\e(X)=\e(S^2)\e(\Sigma_2)+N=-4+N$.  Since every genus-two mapping class is hyperelliptic, the signature formula is Endo's formula.  The remaining identities are the standard almost-complex relations.
\end{proof}

\section{Polizzi's branch model for Xiao's pencil}\label{sec:F2}

We now extract from Polizzi's branch model the geometric information about the thirteen singular fibers that will be used later.  We first identify the six moving branch points on a regular quotient fiber, then determine the local braids at every critical value, and finally formulate the resulting global geometric product relation.  The outcome is the geometric type $(6,7)$ of the fibration and the local Picard--Lefschetz classes.  No explicit Artin word is needed for the purposes of this paper.

\subsection{Polizzi's canonical double cover}

Let $S$ be a surface in Polizzi's locus $\mathcal M_2$, and write its seven reducible genus-two fibers as
\[
G_i=W_i+Z_i,\qquad P_i=W_i\cap Z_i,\qquad 1\le i\le7,
\]
where $W_i$ and $Z_i$ are smooth elliptic curves meeting transversely at one point.  Blow up the seven points $P_i$ and write
\[
\pi:\widehat S\longrightarrow S
\]
for the resulting surface, with exceptional curves $E_i$.  Put $K=K_S$ and
let $G$ denote the class of a fiber of the genus-two pencil.  Polizzi proves
that the movable transform of $|K_S+3G|$ on $\widehat S$ is base-point-free
and induces a generically finite morphism of degree two
\cite[Section~9, Proposition~9.2]{Polizzi}
\[
\widehat f_3:\widehat S\longrightarrow\mathbb F_2
\]
which sends the genus-two pencil to the ruling
\[
p:\mathbb F_2\longrightarrow\mathbb P^1.
\]
Write its Stein factorization as
\[
\widehat S\xrightarrow{c}\overline S\xrightarrow{q_B}\mathbb F_2.
\]
Here $c$ is the canonical resolution and $q_B$ is the finite normal double
cover.  Thus the branch divisor below belongs to $q_B$; the generically
two-to-one map $\widehat f_3=q_B\circ c$ may contract exceptional curves and
is not itself being asserted to be finite.
Let $C_0$ be the negative section, $C_0^2=-2$, and let $L$ denote the class of a ruling fiber.  The branch divisor is
\begin{equation}\label{eq:Polizzi-branch}
B=B^\sharp+\sum_{i=1}^7L_i,
\end{equation}
where the $L_i$ are seven distinct ruling fibers and
\[
B^\sharp\sim 6C_0+15L.
\]
More precisely,
\begin{equation}\label{eq:Polizzi-three-bisections}
B^\sharp=B_1^\sharp+B_2^\sharp+B_3^\sharp,
\qquad
B_j^\sharp\sim2C_0+5L.
\end{equation}
Polizzi's Proposition~9.2 and its multiplicity table give the following
more precise information.  Each $B_j^\sharp$ is irreducible of geometric
genus zero, is linearly equivalent to $2C_0+5L$, and has exactly two
ordinary double points.  The divisor $B^\sharp$ has precisely fourteen
ordinary triple points
\[
x_i,y_i\in L_i,\qquad 1\le i\le7,
\]
two on each distinguished ruling fiber, and no other singularities.  The corresponding singularities of the total branch divisor $B=B^\sharp+\sum L_i$ are ordinary quadruple points, meaning four smooth local branches with distinct tangent directions.  At an $x_{ab}$-type triple point, two local branches come from the node of one bisection and the third from another bisection; at each $m_k,n_k$, the three local branches come from the three distinct bisections.  These statements, including the branch incidence read from the multiplicity table in Proposition~9.2, are input from Polizzi's algebraic construction, not consequences of the monodromy calculation below.  In particular, the ordinary-quadruple local model implies that $L_i$ is transverse to each of the three local branches of $B^\sharp$ at $x_i$ and at $y_i$.

\begin{remark}\label{rem:imported-polizzi}
The divisor class, reducible decomposition, fourteen triple points, and incidence with the seven fibers are quoted from Polizzi's Section~9 (especially Proposition~9.2).  Our contribution in this section begins with the passage from those data to the local spherical braids and their geometric Picard--Lefschetz lifts.  Here the argument does not reconstruct Polizzi's branch curve from scratch.
\end{remark}

\begin{proposition}\label{prop:polizzi-branch}
In the sense of Polizzi's Proposition~9.2, the finite morphism
$q_B:\overline S\to\mathbb F_2$ in the Stein factorization is the normal
double cover branched along \eqref{eq:Polizzi-branch}, and
$c:\widehat S\to\overline S$ is its canonical resolution.  The genus-two
pencil is obtained from the ruling of $\mathbb F_2$, and the seven fibers
$L_i$ contained in the branch locus correspond to the seven reducible
fibers of the pencil.  The fiberwise hyperelliptic involution is induced by
the covering involution.
\end{proposition}

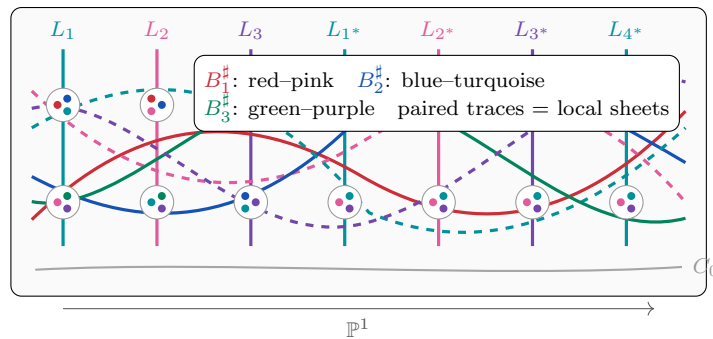
\begin{figure}[ht]
\centering
\begin{tikzpicture}[x=.92cm,y=.92cm,every node/.style={font=\scriptsize}]
  \draw[rounded corners,fill=black!2] (0,0) rectangle (10,4.15);
  \draw[gray!70,thick] (.35,.38) .. controls (3,.52) and (7,.25) .. (9.65,.42)
    node[right,text=gray!80] {$C_0$};
  \foreach \x/\c/\n in {.75/lineteal/{1},2.10/linepink/{2},3.45/linepurple/{3},4.80/lineteal/{1}^{*},6.15/linepink/{2}^{*},7.50/linepurple/{3}^{*},8.85/lineteal/{4}^{*}}{
    \draw[\c,very thick] (\x,.72)--(\x,3.55);
    \node[text=\c,above] at (\x,3.55) {$L_{\n}$};
  }
  \draw[linered,very thick]
    (.30,1.10) .. controls (1.8,2.55) and (3.3,2.75) .. (5.05,1.72)
    .. controls (6.8,.70) and (8.3,1.20) .. (9.70,2.66);
  \draw[linepink,very thick,dashed]
    (.30,2.95) .. controls (1.75,1.55) and (3.4,1.20) .. (5.05,2.22)
    .. controls (6.7,3.22) and (8.25,3.05) .. (9.70,1.35);
  \draw[lineblue,very thick]
    (.30,1.72) .. controls (2.0,.82) and (3.55,1.02) .. (5.15,2.72)
    .. controls (6.55,3.75) and (8.2,2.80) .. (9.70,1.92);
  \draw[lineteal,very thick,dashed]
    (.30,2.42) .. controls (1.9,3.35) and (3.55,3.20) .. (5.15,1.20)
    .. controls (6.7,.55) and (8.2,1.05) .. (9.70,2.42);
  \draw[linegreen,very thick]
    (.30,1.36) .. controls (1.65,1.10) and (3.25,3.35) .. (4.95,3.02)
    .. controls (6.75,2.65) and (8.15,.72) .. (9.70,1.12);
  \draw[linepurple,very thick,dashed]
    (.30,2.70) .. controls (1.7,3.02) and (3.2,.72) .. (4.95,1.02)
    .. controls (6.7,1.35) and (8.15,3.50) .. (9.70,3.08);
  \foreach \x in {.75,2.10,3.45,4.80,6.15,7.50,8.85}{
    \filldraw[fill=white,draw=black!40] (\x,1.35) circle (.24);
    \filldraw[fill=white,draw=black!40] (\x,2.75) circle (.24);
  }
  \fill[linered] (.67,2.75) circle (1.5pt); \fill[lineblue] (.81,2.84) circle (1.5pt); \fill[lineteal] (.81,2.66) circle (1.5pt);
  \fill[linepink] (.67,1.35) circle (1.5pt); \fill[linegreen] (.81,1.44) circle (1.5pt); \fill[linepurple] (.81,1.26) circle (1.5pt);
  \fill[linered] (2.02,2.84) circle (1.5pt); \fill[linepink] (2.02,2.66) circle (1.5pt); \fill[lineblue] (2.17,2.75) circle (1.5pt);
  \fill[lineteal] (2.02,1.35) circle (1.5pt); \fill[linegreen] (2.17,1.44) circle (1.5pt); \fill[linepurple] (2.17,1.26) circle (1.5pt);
  \fill[linered] (3.37,2.84) circle (1.5pt); \fill[linepink] (3.37,2.66) circle (1.5pt); \fill[linegreen] (3.52,2.75) circle (1.5pt);
  \fill[lineblue] (3.37,1.44) circle (1.5pt); \fill[lineteal] (3.37,1.26) circle (1.5pt); \fill[linepurple] (3.52,1.35) circle (1.5pt);
  \foreach \x in {4.80,6.15,7.50,8.85}{
    \fill[linered] (\x-.08,2.75) circle (1.5pt); \fill[lineblue] (\x+.07,2.84) circle (1.5pt); \fill[linegreen] (\x+.07,2.66) circle (1.5pt);
    \fill[linepink] (\x-.08,1.35) circle (1.5pt); \fill[lineteal] (\x+.07,1.44) circle (1.5pt); \fill[linepurple] (\x+.07,1.26) circle (1.5pt);
  }
  \draw[->,black!65] (.75,-.18)--(9.25,-.18)
    node[midway,below] {$\mathbb P^1$};
  \node[draw,fill=white,rounded corners,anchor=north east,align=left]
    at (9.62,3.47) {\textcolor{linered}{$B_1^\sharp$}: red--pink\quad
      \textcolor{lineblue}{$B_2^\sharp$}: blue--turquoise\\[-1pt]
      \textcolor{linegreen}{$B_3^\sharp$}: green--purple\quad
      paired traces $=$ local sheets};
\end{tikzpicture}
\caption{A ruled-chart schematic of Polizzi's branch divisor
$B=B_1^\sharp+B_2^\sharp+B_3^\sharp+\sum_{i=1}^7L_i$ on $\mathbb F_2$.
The six colored traces represent local sheets of the three irreducible
bisections, paired red--pink, blue--turquoise, and green--purple.  The seven
vertical curves are the branch fibers, and the two marked disks on each
record the colored triples of Proposition~\ref{prop:colored-partitions}.
The traces are intentionally interrupted at those disks: the picture
specifies neither global sheet continuation nor a distinguished Hurwitz
ordering.  The thin gray curve is the negative section $C_0$ and is not a
branch component.}
\label{fig:global-branch-locus}
\end{figure}

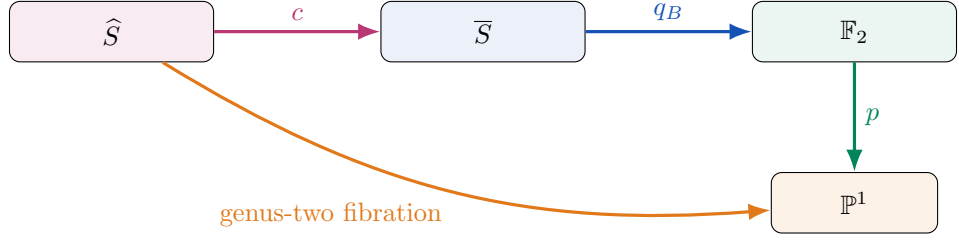
\begin{figure}[ht]
\centering
\begin{tikzpicture}[>=Latex,node distance=1.8cm,every node/.style={font=\small}]
\node[draw,rounded corners,fill=quartic!10,minimum width=2.7cm,minimum height=.8cm] (Shat) {$\widehat S$};
\node[draw,rounded corners,fill=lineblue!8,right=2.2cm of Shat,minimum width=2.7cm,minimum height=.8cm] (Sbar) {$\overline S$};
\node[draw,rounded corners,fill=linegreen!8,right=2.2cm of Sbar,minimum width=2.7cm,minimum height=.8cm] (Ftwo) {$\mathbb F_2$};
\node[draw,rounded corners,fill=lineorange!10,below=1.45cm of Ftwo,minimum width=2.2cm,minimum height=.8cm] (Base) {$\mathbb P^1$};
\draw[->,very thick,quartic] (Shat)--node[above] {$c$} (Sbar);
\draw[->,very thick,lineblue] (Sbar)--node[above] {$q_B$} (Ftwo);
\draw[->,very thick,linegreen] (Ftwo)--node[right] {$p$} (Base);
\draw[->,very thick,lineorange,bend right=18] (Shat) to node[below left] {genus-two fibration} (Base);
\end{tikzpicture}
\caption{Polizzi's canonical double-cover model.  The map $q_B$ is the
normal double cover with branch divisor
$B=B_1^\sharp+B_2^\sharp+B_3^\sharp+\sum_{i=1}^7L_i$, and $c$ is its
canonical resolution.}
\label{fig:polizzi-double-cover}
\end{figure}

\subsection{A marked degree-four reference cover}

The restriction of the Albanese map to a regular genus-two fiber
$C\cong\Sigma_2$ has degree four.  Let $\iota_C$ be the hyperelliptic
involution.  After translating the origin of $E$, the identity
$h\circ\iota_C=-h$ gives the quotient diagram
\begin{figure}[ht]
\centering
\begin{tikzcd}[column sep=4.2em,row sep=3.0em]
C \arrow[r,"h", "4"'] \arrow[d,"q_C"',"2"] &
E \arrow[d,"q_E","2"']\\
\mathbb P^1_x \arrow[r,"\varphi"',"4"] &
\mathbb P^1_z=E/\{\pm1\}.
\end{tikzcd}
\caption{The intrinsic degree-four branch-cover picture on a regular fiber.
The six branch points of $q_C$ occur in three pairs supplied by Polizzi's
three bisections.}
\label{fig:degree-four-quotient}
\end{figure}

The quotient diagram is intrinsic.  Indeed, the hyperelliptic involution
acts as $-1$ on $H^0(C,K_C)$.  Hence the differential of
$h\circ\iota_C+h$ vanishes, so this map is constant.  Translating the
origin of $E$ makes the constant zero and gives
$h\circ\iota_C=-h$.  In particular every Weierstrass point maps to
$E[2]$.

Let $\xi_1,\xi_2,\xi_3$ be the nonzero two-torsion points of $E$.  In
Polizzi's model the bisection $B_j^\sharp$ is the image of divisors
$\xi_j+x+(-x)$, whose Abel--Jacobi class is constantly $\xi_j$.  Therefore
the two points of $B_j^\sharp\cap L_t$ correspond to a pair of Weierstrass
points $w_j^+,w_j^-$ satisfying
\[
\varphi(q_C(w_j^+))=\varphi(q_C(w_j^-))=q_E(\xi_j).
\]
Thus the three bisections pair the six hyperelliptic branch points over the
three nonzero two-torsion values.  This intrinsic pairing is all that is
used below; no marked permutation tuple or Hurwitz-class identification is
asserted.

\begin{figure}[ht]
\centering
\begin{tikzpicture}[x=1cm,y=1cm,every node/.style={font=\small}]
  \draw[rounded corners,fill=black!2] (-4.7,-.95) rectangle (4.7,1.18);
  \node[anchor=west] at (-4.45,.93) {six branch points on $L_t\cong\mathbb P^1_x$};
  \draw[gray!65] (-4.25,-.08)--(4.25,-.08);
  \coordinate (r) at (-3.75,.16);
  \coordinate (p) at (-2.55,-.26);
  \coordinate (b) at (-1.35,.18);
  \coordinate (t) at (.05,-.24);
  \coordinate (g) at (1.55,.19);
  \coordinate (v) at (3.25,-.22);
  \fill[linered] (r) circle (3.8pt) node[above=5pt,text=linered] {$w_1^+$};
  \fill[linepink] (p) circle (3.8pt) node[below=5pt,text=linepink] {$w_1^-$};
  \fill[lineblue] (b) circle (3.8pt) node[above=5pt,text=lineblue] {$w_2^+$};
  \fill[lineteal] (t) circle (3.8pt) node[below=5pt,text=lineteal] {$w_2^-$};
  \fill[linegreen] (g) circle (3.8pt) node[above=5pt,text=linegreen] {$w_3^+$};
  \fill[linepurple] (v) circle (3.8pt) node[below=5pt,text=linepurple] {$w_3^-$};
  \draw[linered!80!linepink,very thick,dashed,bend left=28]
    (r) to node[above,text=black] {$B_1^\sharp$} (p);
  \draw[lineblue!75!lineteal,very thick,dashed,bend left=25]
    (b) to node[above,text=black] {$B_2^\sharp$} (t);
  \draw[linegreen!75!linepurple,very thick,dashed,bend left=24]
    (g) to node[above,text=black] {$B_3^\sharp$} (v);
\end{tikzpicture}
\caption{The intrinsic pairing of the six branch points on a regular ruling
fiber.  Red--pink, blue--turquoise, and green--purple are the two
intersections with $B_1^\sharp,B_2^\sharp,B_3^\sharp$, respectively.  Their
left-to-right placement is schematic: the picture specifies neither a
distinguished path system nor a global Hurwitz ordering.}
\label{fig:colored-branch-pairs}
\end{figure}
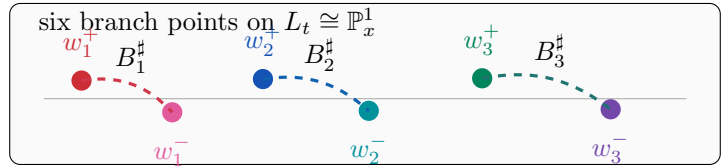

\subsection{The six moving hyperelliptic branch points}

Fix a regular value $t\in\mathbb P^1$ and write $L_t=p^{-1}(t)\cong\mathbb P^1$.  Since
\[
B_j^\sharp\cdot L=(2C_0+5L)\cdot L=2,
\]
each bisection contributes two points to $L_t$.  Thus
\[
L_t\cap B^\sharp=\{b_1(t),\ldots,b_6(t)\}.
\]
After choosing an affine coordinate $u$ on $L_t$, the double cover of the ruling fiber is therefore represented birationally by
\begin{equation}\label{eq:Polizzi-hyperelliptic-fiber}
v^2=\prod_{\nu=1}^6\bigl(u-b_\nu(t)\bigr).
\end{equation}
For regular $t$ this is a smooth genus-two curve.  Polizzi's paper specifies
the divisor classes and incidence data rather than a single affine polynomial
whose six roots can immediately be continued numerically.

\subsection{The six simple ramification values}

Let
\[
\nu_j:\widetilde B_j^\sharp\longrightarrow B_j^\sharp
\]
be the normalization.  Since $B_j^\sharp$ is rational, $\widetilde B_j^\sharp\cong\mathbb P^1$.  The ruling induces a degree-two morphism
\[
\rho_j=p\circ\nu_j:\widetilde B_j^\sharp\longrightarrow\mathbb P^1.
\]
Riemann--Hurwitz gives
\[
-2=2(-2)+\deg R_{\rho_j},
\qquad\text{hence}\qquad
\deg R_{\rho_j}=2.
\]
Thus the three normalizations have total ramification degree six.  Because
each $\rho_j$ has degree two in characteristic zero, every ramification
point has index two and contributes one to $R_{\rho_j}$.  Consequently
there are two simple ramification points on each normalization, six points
on their disjoint union.  The covering involutions can in fact be read
directly from Polizzi's elliptic-curve formulas.

\begin{proposition}\label{prop:bisection-quotient}
Let $\xi_j\in E[2]\setminus\{0\}$.  Under the normalization
\[
\widetilde B_j^\sharp\cong E/\{\pm1\},\qquad
[x]\longmapsto \xi_j+x+(-x),
\]
the deck involution of
$\rho_j:\widetilde B_j^\sharp\to\mathbb P^1$ is
\[
\tau_j([x])=[x+\xi_j].
\]
Equivalently, $\rho_j$ is the quotient by $\tau_j$, followed by an
automorphism of its target $\mathbb P^1$.

More explicitly, write
\[
E:\quad y^2=(x-e_1)(x-e_2)(x-e_3),
\qquad \xi_j=(e_j,0),
\]
and let $\{j,k,\ell\}=\{1,2,3\}$.  On the coordinate
$x:E/\{\pm1\}\to\mathbb P^1$, the involution is
\[
x\longmapsto e_j+
\frac{(e_j-e_k)(e_j-e_\ell)}{x-e_j}.
\]
Thus an invariant degree-two coordinate is
\[
q_j(x)=x+e_j+
\frac{(e_j-e_k)(e_j-e_\ell)}{x-e_j},
\]
and the two ramification points are characterized by
\[
(x-e_j)^2=(e_j-e_k)(e_j-e_\ell).
\]
\end{proposition}

\begin{proof}
Choose generators $\eta_1,\ldots,\eta_4$ for the four order-three cyclic
subgroups of $E[3]$; these subgroups correspond bijectively to the four
inversion orbits $\{\pm\eta_r\}$.  The following formulas are read from
Polizzi's descriptions of $A_{\xi_j},M_r,N_r$ and the reducible fibers
\cite[Sections~5 and~9]{Polizzi}.  Polizzi's four reducible fibers
$G_r^*=M_r+N_r$ meet $A_{\xi_j}$ in the two divisors
\[
\xi_j+\eta_r+2\eta_r,
\qquad
\xi_j+(\xi_j+\eta_r)+(\xi_j+2\eta_r).
\]
On the normalization these are the points $[\eta_r]$ and
$[\eta_r+\xi_j]$, respectively.  Polizzi's transverse intersection
calculation shows that both occur with multiplicity one.  They are therefore
the two reduced intersections
with $G_r^*$ and hence the entire degree-two fiber of $\rho_j$ over that
value.  Translation by $\xi_j$ commutes with inversion
and therefore descends to the involution $\tau_j$ of $E/\{\pm1\}$.
The four points $[\eta_r]$ are distinct in $E/\{\pm1\}$.  The deck
involution sends each $[\eta_r]$ to $[\eta_r+\xi_j]$, exactly as $\tau_j$
does.  Since two automorphisms of $\mathbb P^1$ agreeing on three distinct
points coincide, the deck involution is $\tau_j$.

The addition law on the Weierstrass model gives the displayed fractional
linear formula for translation by $(e_j,0)$.  Adding $x$ produces an
invariant rational function of degree two.  The ramification points of the
quotient are the fixed points of the involution, which satisfy the final
equation.
\end{proof}

The three target automorphisms implicit in
Proposition~\ref{prop:bisection-quotient} need not agree.  Consequently the
formulas locate the ramification points on each normalized bisection, but
do not by themselves place all six branch values in a single marked
coordinate on the pencil base.  We next show, using the singular-fiber
classification, that their images in the common base are pairwise
distinct.

At a point lying over one of the distinguished fibers $L_i$, the
ordinary-quadruple description of $B=B^\sharp+L_i$ implies that the ruling
fiber is transverse to every local branch of $B^\sharp$.  Hence none of
these points is a ramification point of the corresponding normalization
map.  All six ramification points occur away from the seven $L_i$.  By
\cite[Proposition~9.2]{Polizzi}, the singularities of $B^\sharp$ occur on
the distinguished fibers.  Away from those fibers, a ramification point is
therefore a smooth tangency of one normalized bisection to the ruling.  The
local double-cover equation is an $A_1$ model and produces one nodal fiber.
Riemann--Hurwitz gives six ramification points counted with multiplicity,
while Xiao's classification, recalled in
\cite[Proposition~6.1]{Polizzi}, gives exactly six irreducible one-node
fibers away from the $L_i$.  If two ramification points had the same ruling
value, the corresponding fiber would contain two nodes, or a non-transverse
collision would give a worse singularity.  Either alternative contradicts
that classification.  The six ramification values are therefore pairwise
distinct and exhaust the six irreducible nodal fibers.

\begin{proposition}\label{prop:polizzi-six-nonseparating}
Let $\tau$ be one of the six ramification values.  Exactly two of the six points in \eqref{eq:Polizzi-hyperelliptic-fiber} coalesce at $\tau$.  For a positively oriented local coordinate $s=t-\tau$, suitable fiber coordinates give
\[
u_\pm(s)=u_0\pm c\sqrt{s}+O(s),\qquad c\ne0.
\]
A positive circuit around $\tau$ therefore gives a positive half twist $h_{\alpha_\tau}$ about an arc joining the two colliding branch points.  Its hyperelliptic lift is the positive Dehn twist
\[
t_{a_\tau},\qquad a_\tau=q^{-1}(\alpha_\tau),
\]
where $q:\Sigma_2\to S^2$ is the hyperelliptic quotient.  The curve $a_\tau$ is nonseparating.
\end{proposition}

\begin{proof}
A ramification point of a degree-two map in characteristic zero is simple.  The local form is $s=w^2$ after holomorphic coordinate changes, giving the displayed square-root exchange of the two local sheets.  The positive loop in $s$ is the positive geometric half twist.  The inverse image under the genus-two hyperelliptic cover of an arc joining two branch points is a nonseparating simple closed curve, and the half twist lifts to the positive Dehn twist about that curve.
\end{proof}

We denote the resulting six nonseparating vanishing cycles by
\[
v_1,\ldots,v_6.
\]
No labeling by five-chain words is intended at this stage; the labels record the six geometric critical values.

\subsection{The seven \texorpdfstring{$3+3$}{3+3} degenerations}

Fix a distinguished branch fiber $L_i$ and choose a local base coordinate
$s$ with $L_i=p^{-1}(0)$.  Choose a fiber coordinate $u$ near the ordinary
triple point $x_i$.  Because $B^\sharp$ has an ordinary triple point at
$x_i$ and the fourth branch $L_i$ is transverse to all three branches,
after multiplication by a holomorphic unit a local equation is
\[
 \prod_{k=1}^3\bigl(u-x_i-\alpha_{i,k}s+O(s^2)\bigr)=0,
\]
where transversality to $L_i$ makes the slopes finite and ordinariness makes
them pairwise distinct.  Thus the branches have expansions
\begin{equation}\label{eq:Polizzi-triple-x}
u_{i,k}(s)=x_i+\alpha_{i,k}s+O(s^2),\qquad k=1,2,3,
\end{equation}
with pairwise distinct $\alpha_{i,k}$.  Likewise, near the second triple point $y_i$,
\begin{equation}\label{eq:Polizzi-triple-y}
u_{i,k+3}(s)=y_i+\beta_{i,k}s+O(s^2),\qquad k=1,2,3,
\end{equation}
with pairwise distinct $\beta_{i,k}$.  Thus as $s\to0$ the six branch points split into two clusters
\[
\{x_i,x_i,x_i\}\ \sqcup\ \{y_i,y_i,y_i\}.
\]
Since each triple point is ordinary, its three tangent directions are
distinct.  Since $L_i$ is transverse to every branch, those directions have
finite, pairwise distinct slopes with respect to the ruling coordinate.

Polizzi's multiplicity table determines which bisections supply the three
branches in each cluster.  We record this information because it is the
part of the colored branch data that is independent of any choice of
distinguished paths.

\begin{proposition}\label{prop:colored-partitions}
Write only the bisection indices in each triple.  For the three special
fibers one has
\[
\begin{array}{c|c}
L_1 & \{1,2,2\}\mid\{1,3,3\}\\
L_2 & \{1,1,2\}\mid\{2,3,3\}\\
L_3 & \{1,1,3\}\mid\{2,2,3\},
\end{array}
\]
where the two sides correspond respectively to
$(x_{23},x_{32})$, $(x_{13},x_{31})$, and $(x_{12},x_{21})$.
For each of the remaining four fibers $L_r^*$, the two triple points
$m_r,n_r$ both have the transversal type
\[
\{1,2,3\}\mid\{1,2,3\}.
\]
\end{proposition}

\begin{proof}
At $x_{ab}$ the multiplicity of $B_j^\sharp$ is the entry in Polizzi's
multiplicity table.  Reading its six special columns gives the first three
partitions.  In each of the columns $m_r$ and $n_r$, all three entries are
one, giving the transversal partitions.
\end{proof}

\begin{figure}[ht]
\centering
\begin{tikzpicture}[x=1cm,y=1cm,every node/.style={font=\scriptsize}]
\newcommand{\panel}[3]{%
  \draw[rounded corners,fill=black!2] (#1-1.8,#2-.62) rectangle (#1+1.8,#2+.62);
  \draw[gray!65] (#1,#2-.48)--(#1,#2+.38);
  \node at (#1,#2+.82) {#3};}
\panel{-2.25}{1.05}{$L_1:\ (122)\mid(133)$}
\fill[linered] (-3.40,1.05) circle (2.6pt);
\fill[lineblue] (-3.12,1.20) circle (2.6pt);
\fill[lineteal] (-3.12,.90) circle (2.6pt);
\fill[linepink] (-1.10,1.05) circle (2.6pt);
\fill[linegreen] (-1.38,1.20) circle (2.6pt);
\fill[linepurple] (-1.38,.90) circle (2.6pt);
\panel{2.25}{1.05}{$L_2:\ (112)\mid(233)$}
\fill[linered] (1.10,1.20) circle (2.6pt);
\fill[linepink] (1.10,.90) circle (2.6pt);
\fill[lineblue] (1.38,1.05) circle (2.6pt);
\fill[lineteal] (3.40,1.05) circle (2.6pt);
\fill[linegreen] (3.12,1.20) circle (2.6pt);
\fill[linepurple] (3.12,.90) circle (2.6pt);
\panel{-2.25}{-1.45}{$L_3:\ (113)\mid(223)$}
\fill[linered] (-3.40,-1.30) circle (2.6pt);
\fill[linepink] (-3.40,-1.60) circle (2.6pt);
\fill[linegreen] (-3.12,-1.45) circle (2.6pt);
\fill[lineblue] (-1.38,-1.30) circle (2.6pt);
\fill[lineteal] (-1.38,-1.60) circle (2.6pt);
\fill[linepurple] (-1.10,-1.45) circle (2.6pt);
\panel{2.25}{-1.45}{$L_r^*:\ (123)\mid(123)$}
\fill[linered] (1.10,-1.45) circle (2.6pt);
\fill[lineblue] (1.38,-1.30) circle (2.6pt);
\fill[linegreen] (1.38,-1.60) circle (2.6pt);
\fill[linepink] (3.40,-1.45) circle (2.6pt);
\fill[lineteal] (3.12,-1.30) circle (2.6pt);
\fill[linepurple] (3.12,-1.60) circle (2.6pt);
\end{tikzpicture}
\caption{The colored incidence types of the seven $3+3$ fibers.  The
red--pink, blue--turquoise, and green--purple pairs belong to
$B_1^\sharp,B_2^\sharp,B_3^\sharp$, respectively.  The lower-right panel
occurs for each $r=1,\ldots,4$.  Positions within the disks are schematic;
the figure records incidence only, not parallel transport or braid order.}
\label{fig:colored-partitions}
\end{figure}
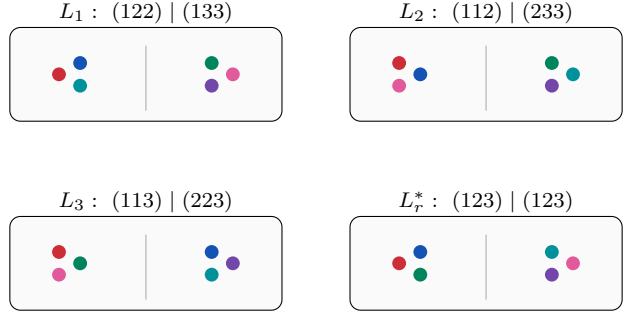

Proposition~\ref{prop:colored-partitions} fixes the unordered local
partitions.  It does not identify the four transversal separating curves
with one another in a marked reference sphere: that comparison still
requires parallel transport along specified paths and the global
clutching data.

The following elementary scaling statement records the local mechanism.

\begin{lemma}\label{lem:polizzi-scaling}
Suppose three points in a disk have expansions
\[
z_k(s)=z_0+s\gamma_k+O(s^2),\qquad k=1,2,3,
\]
with the $\gamma_k$ pairwise distinct.  For sufficiently small $|s|$, a positive circuit of $s$ around zero gives one positive full twist on the three points.
\end{lemma}

\begin{proof}
For small $|s|$ the $O(s^2)$ terms can be removed through a homotopy in the configuration space of three distinct points.  The resulting motion $z_0+s\gamma_k$ is a rigid positive $2\pi$ rotation of the three-point configuration, hence the central full twist $\Delta_3^2$.
\end{proof}

Apply the lemma to \eqref{eq:Polizzi-triple-x} and
\eqref{eq:Polizzi-triple-y}.  Choose complementary disks around the two
clusters.  Their boundary curves are distinct as embedded curves but
represent the same unoriented isotopy class $\delta_i$ on the
six-punctured sphere.  Each cluster contributes the positive twist
$T_{\delta_i}$, so the total spherical monodromy is
\begin{equation}\label{eq:Polizzi-spherical-33}
T_{\delta_i}^2.
\end{equation}

\begin{figure}[ht]
\centering
\begin{tikzpicture}[scale=.88,every node/.style={font=\small}]
\begin{scope}[xshift=-3.6cm]
\draw[rounded corners,fill=lineblue!5] (-2,-1.15) rectangle (2,1.15);
\fill[linered] (-1.25,.25) circle (3pt);
\fill[lineblue] (-.95,-.35) circle (3pt);
\fill[linegreen] (-.65,.4) circle (3pt);
\fill[linepink] (.65,-.25) circle (3pt);
\fill[lineteal] (.95,.38) circle (3pt);
\fill[linepurple] (1.3,-.38) circle (3pt);
\draw[linegreen,very thick] (-.95,0) ellipse (.72 and .72);
\node at (0,-1.45) {regular nearby fiber};
\node[linegreen] at (0,1.42) {$\delta_i$ separates $3+3$};
\end{scope}
\draw[->,very thick] (-.8,0)--(.8,0) node[midway,above] {$s\to0$};
\begin{scope}[xshift=3.6cm]
\draw[rounded corners,fill=lineblue!5] (-2,-1.15) rectangle (2,1.15);
\fill[linered] (-1.05,0) circle (3pt);
\fill[lineblue] (-.9,.12) circle (3pt);
\fill[linegreen] (-.9,-.12) circle (3pt);
\fill[linepink] (.9,0) circle (3pt);
\fill[lineteal] (1.05,.12) circle (3pt);
\fill[linepurple] (1.05,-.12) circle (3pt);
\draw[linegreen,very thick] (-.98,0) ellipse (.48 and .42);
\node at (0,-1.45) {$3+3$ collision on $L_r^*$};
\end{scope}
\end{tikzpicture}
\caption{One of the four transversal $3+3$ degenerations in the $K^2=3$ pencil.
A positive loop around the distinguished ruling fiber rotates both triples;
the spherical monodromy is $T_{\delta_i}^2$ and its geometric Lefschetz
lift is the separating twist $t_{d_i}$.  The three special incidence types
are displayed in Figure~\ref{fig:colored-partitions}.}
\label{fig:polizzi-33}
\end{figure}
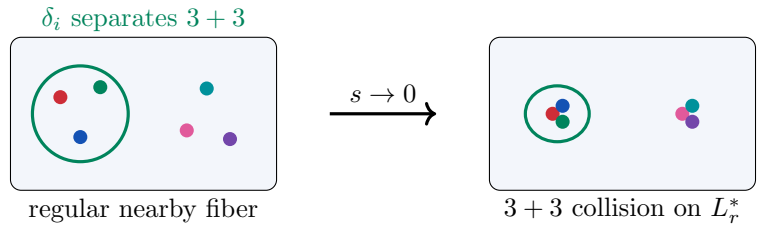
\begin{remark}\label{rem:vertical-branch}
At $s=0$ the fiber $L_i$ itself is a component of the branch divisor, so one should not interpret the six-point equation for a regular ruling fiber as a finite double-cover equation obtained by simply substituting $s=0$.  The braid calculation is performed over the punctured disk $0<|s|\ll1$, where the branch on the ruling fiber consists of the six intersections with $B^\sharp$.  The canonical resolution of the double cover then supplies the central fiber.  Polizzi identifies that central fiber, after returning to the relatively minimal genus-two fibration, as $W_i+Z_i$.  Consequently the spherical $3+3$ braid and the Picard--Lefschetz nodal model are two descriptions of the same degeneration, but the latter is what fixes the lift in $\Mod(\Sigma_2)$.
\end{remark}

\begin{proposition}\label{prop:polizzi-separating-lift}
Let $d_i=q^{-1}(\delta_i)\subset\Sigma_2$.  Then $d_i$ is connected and separating, and the actual local Lefschetz monodromy of Xiao's $K^2=3$ pencil around $L_i$ is the single positive twist
\[
t_{d_i}.
\]
It maps to $T_{\delta_i}^2$ under the Birman--Hilden quotient
\cite{BirmanHilden}.
\end{proposition}

\begin{proof}
The curve $\delta_i$ separates the six branch points into two sets of three.  The double cover of a disk branched at three interior points is a one-holed torus: its Euler characteristic is $2-3=-1$, and its boundary is connected because the number of branch points is odd.  Consequently $q^{-1}(\delta_i)=d_i$ is connected and separates $\Sigma_2$ into two one-holed tori.

Over a punctured neighborhood of the distinguished value, the three branch points near $x_i$ undergo one positive full twist and the three near $y_i$ undergo another.  The boundaries of the two complementary three-punctured disks represent the same unoriented isotopy class $\delta_i$, and both motions have positive complex orientation.  Hence the spherical monodromy is $T_{\delta_i}^2$.

This does not by itself determine the Lefschetz lift, because the Birman--Hilden projection has the hyperelliptic involution in its kernel.  The lift is fixed by the local geometry of the singular fiber.  Polizzi's corresponding fiber is $W_i+Z_i$, where $W_i$ and $Z_i$ are smooth elliptic curves meeting transversely at one point.  Near that point the holomorphic map has, after analytic changes of coordinates, the ordinary nodal form
\[
xy=s.
\]
Picard--Lefschetz theory therefore gives one positive Dehn twist about the neck separating the two genus-one components.  This neck is precisely $d_i$.  Thus the geometric local monodromy is $t_{d_i}$.  Since the quotient map has angular degree two on an annular neighborhood of $d_i$, the twist $t_{d_i}$ projects to $T_{\delta_i}^2$.
\end{proof}

\subsection{The global monodromy relation}

Let
\[
\mathcal C=\{\tau_1,\ldots,\tau_6,\lambda_1,\ldots,\lambda_7\}\subset\mathbb P^1
\]
be the six simple ramification values and the seven values corresponding to the branch fibers $L_i$.  Choose a regular value $t_*$ and a distinguished geometric basis of positively oriented meridians
\[
\gamma_1,\ldots,\gamma_{13}
\]
for $\mathbb P^1\setminus\mathcal C$, with the product ordered according to the chosen star system:
\[
\gamma_1\cdots\gamma_{13}=1.
\]
Parallel transport identifies the smooth fibers with one labeled reference surface $\Sigma_2$ and gives
\[
\rho:\pi_1(\mathbb P^1\setminus\mathcal C,t_*)\longrightarrow\Mod(\Sigma_2).
\]
Put $M_k=\rho(\gamma_k)$.
Products are read with the rightmost mapping class acting first, and the
meridians and monodromy representation are chosen consistently with this
convention.

\begin{theorem}\label{thm:polizzi-geometric-factorization}
For the genus-two fibration determined by Polizzi's branch pair, and for
every distinguished basis chosen with the stated product convention,
\begin{equation}\label{eq:Polizzi-global-product}
M_1M_2\cdots M_{13}=1\qquad\text{in }\Mod(\Sigma_2).
\end{equation}
Exactly six factors are positive Dehn twists about nonseparating curves and exactly seven are positive Dehn twists about separating curves.  More precisely, the former are the lifts of the six simple ramification half twists of the three bisections, and the latter are the geometric lifts $t_{d_i}$ of the seven $3+3$ spherical monodromies $T_{\delta_i}^2$.  This is a geometric monodromy relation of type $(6,7)$; no explicit marked Hurwitz factorization is asserted.
\end{theorem}

\begin{proof}
The product relation follows by applying the monodromy representation $\rho$ to the standard product relation for a distinguished basis of meridians in the punctured sphere.  The local factors over the six $\tau_j$ are given by \cref{prop:polizzi-six-nonseparating}.  The local factors over the seven $\lambda_i$ are given by \cref{prop:polizzi-separating-lift}.  Positivity follows from the complex orientation and the holomorphic Picard--Lefschetz models.
\end{proof}

Thus local braid and Picard--Lefschetz theory applied to Polizzi's branch
model, together with Xiao's singular-fiber classification as recalled by
Polizzi, realizes
\[
(n,s)=(6,7)
\]
and explains the topology of all thirteen local factors.  If one suppresses the interleaving dictated by the chosen paths, the set of factors may be written schematically as
\[
\{t_{v_1},\ldots,t_{v_6},t_{d_1},\ldots,t_{d_7}\};
\]
this notation is \emph{not} an assertion that all six nonseparating twists occur before all seven separating twists in a global word.

\subsection{Numerical consistency checks}

The branch-derived type gives
\[
\e(S)=4-4\cdot2+(6+7)=9.
\]
Endo's genus-two hyperelliptic signature formula \cite{Endo} gives
\[
\sigma(S)=-\frac35\cdot6-\frac15\cdot7=-5,
\]
and therefore
\[
\chi_h(S)=\frac{9-5}{4}=1,
\qquad
c_1^2(S)=2\cdot9+3(-5)=3.
\]
These are exactly the invariants of the Xiao surface in Polizzi's model.  The calculation is a consistency check, not the proof of the monodromy: the geometric proof is \cref{thm:polizzi-geometric-factorization}.

\section{The fundamental group of Xiao's surface in Polizzi's model}
\label{sec:pi1-polizzi}

Polizzi's $E(3)$ model also determines the fundamental group directly,
without waiting for the thirteen vanishing cycles.  This gives
an independent check on any proposed topological construction of the same
surface.

\subsection{The ambient threefold}

Let
\[
\beta:E(3)\longrightarrow E
\]
be the Abel--Jacobi map.  As in the Catanese--Ciliberto construction and in
Polizzi's notation,
\[
E(3)=\mathbb P(\mathcal E_3),
\]
where $\mathcal E_3$ is the Atiyah bundle of rank $3$ and degree $1$ on the
elliptic curve $E$ \cite{Atiyah,CataneseCiliberto,Polizzi}.  We use the
quotient projectivization convention.  Its fiber is $\mathbb P^2$.
Consequently the homotopy
exact sequence of the locally trivial projective bundle gives
\begin{equation}\label{eq:pi1-E3}
\pi_1(E(3))
\cong
\pi_1(E)
\cong
\mathbb Z^2.
\end{equation}
This agrees with the fact that the Albanese variety of every surface in the
family is $E$.

\subsection{Ampleness of the class \texorpdfstring{$4D-F$}{4D-F}}

Let $D$ denote the tautological divisor class and $F$ the fiber class of
$\beta$, so $F=\beta^*(p)$ for a point $p\in E$.  The Catanese--Ciliberto construction gives
\[
K_{E(3)}\sim -3D+F
\]
and identifies the relevant canonical models with divisors algebraically
equivalent to $4D-F$ \cite{CataneseCiliberto}.  Here $4D-F$ denotes the
corresponding class in $\operatorname{NS}(E(3))$; a chosen line bundle
$\mathfrak D_0$ represents that class.  Polizzi's Theorem~7.1 then
identifies the genus-two-pencil member as the smooth divisor
\[
S_E=X_E\in|\mathfrak D_0|,
\qquad
\mathfrak D_0\equiv4D-F.
\]

\begin{lemma}\label{lem:ample-4D-F}
The divisor class $4D-F$ on $E(3)=\mathbb P(\mathcal E_3)$ is ample.
\end{lemma}

\begin{proof}
The Atiyah bundle $\mathcal E_3$ has rank $3$ and degree $1$.  It is
indecomposable, and since $\gcd(3,1)=1$ it is stable \cite{Atiyah}; hence
its minimal Harder--Narasimhan slope is
$\mu^{-}(\mathcal E_3)=1/3$.  We use Grothendieck's
quotient convention
$\mathbb P(\mathcal E_3)=\operatorname{Proj}(\operatorname{Sym}\mathcal E_3)$,
as in Polizzi, so that
$D=c_1(\mathcal O_{\mathbb P(\mathcal E_3)}(1))$.  Miyaoka's ample-cone
criterion for a projective bundle over a curve, quoted in
\cite[Theorem~2.1]{AlzatiBertoliniBesana}, says that a divisor
\[
aD+\beta^*M,\qquad \deg M=b,
\]
is ample precisely when
\[
a>0,\qquad b+a\,\mu^{-}(\mathcal E_3)>0.
\]
For $4D-F$ one has $a=4$ and $b=-1$, so
\[
-1+\frac43=\frac13>0.
\]
Hence $4D-F$ is ample.
\end{proof}

As a check,
\[
D^3=\deg\mathcal E_3=1,\qquad D^2F=1,\qquad F^2=0,
\]
so
\[
(4D-F)^3=64-48=16>0.
\]
Adjunction also gives
\[
K_{S_E}
=(K_{E(3)}+S_E)|_{S_E}
=D|_{S_E}
\]
and therefore
\[
K_{S_E}^2
=D^2(4D-F)=4-1=3,
\]
as required.

\subsection{The fundamental group of \texorpdfstring{$S_E$}{S E}}

\begin{theorem}\label{thm:pi1-polizzi}
For every surface in Xiao's one-dimensional $K^2=3$ family described by
Polizzi,
\[
\pi_1(S_E)\cong\mathbb Z^2.
\]
More precisely, the inclusion
\[
j:S_E\hookrightarrow E(3)
\]
induces an isomorphism
\[
j_*:\pi_1(S_E)\xrightarrow{\;\cong\;}\pi_1(E(3)),
\]
and under \eqref{eq:pi1-E3} this is the fundamental-group map induced by
the Albanese morphism.
\end{theorem}

\begin{proof}
Polizzi proves in \cite[Theorem~7.1]{Polizzi} that the divisor
$X_E\subset E(3)$ associated with a surface in the genus-two-fibration locus
is smooth.
By \cref{lem:ample-4D-F}, $X_E$ is a smooth ample divisor in the smooth
complex projective threefold $E(3)$.  Thus
$\mathcal O_{E(3)}(X_E)$ is a positive line bundle and $X_E$ is the
nonsingular zero locus of a section.  Bott's homotopy form of the
topological Lefschetz hyperplane theorem
\cite[Theorem~I]{BottLefschetz} gives
\[
\pi_i(E(3),X_E)=0\qquad (i<3).
\]
The relative homotopy exact sequence therefore gives an isomorphism
\[
\pi_1(X_E)\xrightarrow{\cong}\pi_1(E(3)).
\]
Equation \eqref{eq:pi1-E3} identifies the latter group with
$\pi_1(E)\cong\mathbb Z^2$.  Since the fiber of
$\beta:E(3)\to E$ is $\mathbb P^2$, the map $\beta_*$ is an isomorphism.
Moreover $\alpha=\beta|_{S_E}$, so the displayed isomorphism is precisely
the map induced by the Albanese morphism.
\end{proof}

The cusp calculation has the following explicit homological normal form.
For a primitive vector $(a,b)$ in the basis $(p,q)$, write
\[
T_{(a,b)}=
\begin{pmatrix}
1-4ab&4a^2\\
-4b^2&1+4ab
\end{pmatrix}.
\]
Direct multiplication gives
\[
T_{(1,0)}T_{(1,-1)}T_{(0,1)}T_{(1,2)}T_{(1,1)}T_{(2,1)}=I.
\]
The six slopes reduce to the six elements of
$\mathbb P^1(\mathbb Z/4)$.  This identity is only a normal form in
$\operatorname{Aut}(K_0)$; it neither identifies six marked curves on
$\Sigma_2$ nor determines the seven separating cycles or their
interleaving in a Hurwitz factorization.

\begin{remark}
Polizzi does not state this fundamental-group computation explicitly; the
argument above is a consequence of his realization of the surface as the
smooth divisor $X_E\subset E(3)$ together with the classical Lefschetz
hyperplane theorem.  In particular, the conclusion is stronger than the
formal consequence $q=1\Rightarrow b_1=2$: it determines the full
nonabelian fundamental group and shows that there is no additional finite
kernel over the Albanese group.
\end{remark}

\subsection{The fundamental-group quotient from the genus-two pencil}

It is useful to record the van Kampen quotient because it separates the
relations supplied by the Lefschetz thimbles from the additional information
supplied by the Albanese morphism.
Polizzi's pencil $|G|$ is base-point free, so it is an honest genus-two
fibration
\[
f_G:S_E\longrightarrow\mathbb P^1
\]
with thirteen singular fibers.  Choose a regular value $b_0$, a marking
\[
F_{b_0}\cong\Sigma_2,
\]
and distinguished paths to the thirteen critical values.  Let
\[
\Gamma_2=
\langle a_1,b_1,a_2,b_2\mid
[a_1,b_1][a_2,b_2]=1\rangle
\]
and let
\[
\delta_1,\ldots,\delta_{13}\in\Gamma_2
\]
be based representatives of the vanishing cycles.  Choose a disk containing
all critical values.  Its inverse image is obtained from
$\Sigma_2\times D^2$ by attaching the Lefschetz $2$-handles, so its
fundamental group is generated by the regular fiber and the thimbles kill
the normal closure of the vanishing cycles.  After choosing a bundle
trivialization over the complementary disk, its inverse image is
$\Sigma_2\times D^2$.  The boundary gluing may carry a pointed clutching
element, which can add relations but no generators.  Van Kampen therefore
shows that inclusion of the regular fiber surjects onto $\pi_1(S_E)$.
Hence there is a
natural epimorphism
\begin{equation}\label{eq:pi1-polizzi-presentation}
\Gamma_2/
\langle\!\langle
\delta_1,\ldots,\delta_{13}
\rangle\!\rangle
\twoheadrightarrow
\pi_1(S_E)\cong\mathbb Z^2.
\end{equation}
Equality would follow, for example, from control of the pointed clutching
class or from a suitable section of the genus-two fibration; neither is
assumed here.  After
abelianization, \eqref{eq:pi1-polizzi-presentation} gives a surjection from
the quotient of $H_1(\Sigma_2;\mathbb Z)$ by the six nonseparating vanishing
classes onto $H_1(S_E;\mathbb Z)\cong\mathbb Z^2$.  The stronger lattice
statement used below comes instead from the restriction of the Albanese
morphism to a regular fiber.

\subsection{A fundamental-group check for alternative constructions}

Suppose an alternative construction produces a genus-two fibration $Y$ of
type $(6,7)$.  Equality of numerical invariants is not enough.  A necessary
condition for $Y$ to be Xiao's $K^2=3$ surface is
\[
\pi_1(Y)\cong\mathbb Z^2.
\]
Thus the algebraic computation in \cref{thm:pi1-polizzi} gives a strong
independent test on any replacement construction.  If a proposed
cut-and-paste description has a computable complement presentation, its
fundamental group should be compared with this target before attempting the
finer branch-pair identification.

\begin{lemma}\label{lem:polizzi-disk-piece}
Let $D\subset S^2$ be a closed disk containing all thirteen critical values
and put $Y=f^{-1}(D)$.  If $v_1,\ldots,v_{13}$ are the vanishing cycles in a
regular boundary fiber $F$, then
\[
\pi_1(Y)\cong
\pi_1(F)/\langle\!\langle v_1,\ldots,v_{13}\rangle\!\rangle,
\]
and filling by the complementary disk induces an epimorphism
\[
\pi_1(Y)\twoheadrightarrow\pi_1(S_E)\cong\mathbb Z^2.
\]
\end{lemma}

\begin{proof}
A Lefschetz fibration over a disk is obtained from $F\times D^2$ by attaching
one $2$-handle for each vanishing cycle, with Lefschetz framing.  Van Kampen
therefore kills the normal closure of the vanishing cycles and introduces no
new generators.  Filling by the complementary disk recovers $S_E$, so the
second assertion follows from van Kampen and \cref{thm:pi1-polizzi}.
\end{proof}

\section{The Albanese-kernel lattice of Xiao's
\texorpdfstring{$K^2=3$}{K2=3} pencil}
\label{sec:polizzi-vanishing-lattice}

Let $\Sigma$ be a regular genus-two fiber and let
\[
v_1,\ldots,v_6\in H_1(\Sigma;\mathbb Z)
\]
be the homology classes of the six nonseparating vanishing cycles.
Let
\[
h=\alpha|_\Sigma:\Sigma\longrightarrow E,
\qquad
K_0=\ker\{h_*:H_1(\Sigma;\mathbb Z)\to H_1(E;\mathbb Z)\}.
\]
The following calculation uses only the degree-four Albanese map and the
global monodromy relation; it does not require an explicit word for the six
vanishing cycles.

\begin{lemma}\label{lem:primitive-complement}
Let $H$ be a unimodular symplectic lattice of rank four and let
$L\subset H$ be a primitive nondegenerate sublattice of rank two.  Then
$L^\perp$ is primitive, and
\[
|\det L^\perp|=|\det L|.
\]
\end{lemma}

\begin{proof}
Primitivity of $L^\perp$ follows because it is the kernel of the homomorphism
$H\to L^*$ induced by the intersection form.  The unimodular pairing embeds
$L\oplus L^\perp$ as a finite-index sublattice of $H$.  Comparing
determinants, or equivalently the two discriminant groups obtained from this
embedding, gives an isomorphism of their underlying finite groups and hence
the stated equality.
\end{proof}

\begin{lemma}\label{lem:gamma-four}
Let
\[
\Gamma(4)=\ker\{SL(2,\mathbb Z)\longrightarrow
SL(2,\mathbb Z/4)\}.
\]
Let $\overline\Gamma(4)$ be its image in $PSL(2,\mathbb Z)$.  This image is
torsion free.  The compact modular curve $X(4)$ has genus zero, and
$Y(4)=\overline\Gamma(4)\backslash\mathbb H$ is its complement of six
cusps.  Since $-I\notin\Gamma(4)$, the projection identifies
$\Gamma(4)$ with $\overline\Gamma(4)$.  The
cusps are indexed by the primitive vectors of $(\mathbb Z/4)^2$ up to sign,
namely by $\mathbb P^1(\mathbb Z/4)$, here defined as unimodular pairs
modulo multiplication by the units $\{\pm1\}$.  If $v=(a,b)$ is primitive, the
positive transvection for the alternating form $4J$ is a primitive positive
peripheral element at the cusp $[a:b]$.  If $e_1,\ldots,e_6$ denote the
oriented peripheral classes, then
\[
H_1(Y(4);\mathbb Z)=
\langle e_1,\ldots,e_6\mid e_1+\cdots+e_6=0\rangle .
\]
\end{lemma}

\begin{proof}
The reduction map $SL(2,\mathbb Z)\to SL(2,\mathbb Z/4)$ is surjective and
$|SL(2,\mathbb Z/4)|=48$.  Since $-I\notin\Gamma(4)$, the image
$\overline\Gamma(4)\subset PSL(2,\mathbb Z)$ has index $24$ and is torsion
free.  Its cusps are represented by primitive vectors of
$(\mathbb Z/4)^2$ modulo sign, giving $(16-4)/2=6$ classes.  The stabilizer
of infinity is generated by
$\left(\begin{smallmatrix}1&4\\0&1\end{smallmatrix}\right)$, and normality
shows that every cusp has width four.  The standard genus formula gives
$X(4)\cong\mathbb P^1$; see
\cite[\S\S2.1--2.4 and~3.1]{DiamondShurman}.  Thus
$\pi_1(Y(4))\cong\overline\Gamma(4)\cong\Gamma(4)$.

If $A\in SL(2,\mathbb Z)$ has first column $(a,b)^T$, then
\[
A\begin{pmatrix}1&4\\0&1\end{pmatrix}A^{-1}
=\begin{pmatrix}1-4ab&4a^2\\-4b^2&1+4ab\end{pmatrix}.
\]
This is exactly the transvection for the form $4J$ in the convention used
below.  We orient each peripheral loop so that this primitive stabilizer is
positive.  The final presentation is the first homology of a six-punctured
sphere.
\end{proof}

\begin{theorem}\label{thm:polizzi-vanishing-lattice}
The Albanese kernel $K_0$ is primitive of rank two.  In a suitable basis
$p,q$ its
intersection form is
\[
p\cdot q=4,
\qquad
\left.K_0\right|_{\cdot}\cong
\begin{pmatrix}0&4\\-4&0\end{pmatrix}.
\]
Moreover, after reordering the six nonseparating vanishing cycles, their
classes lie in $K_0$, and their primitive slopes in $K_0$ reduce modulo $4$
to the six distinct elements of
$\mathbb P^1(\mathbb Z/4)$:
\[
[1:0],\ [0:1],\ [1:1],\ [1:2],\ [1:3],\ [2:1].
\]
Here a matching path means a path whose two transported vanishing cycles
agree as unoriented curves.  In particular, no two of the six nonseparating vanishing cycles are
homologous up to sign.  Hence no two of the six nonseparating critical
points can be joined by a matching path.
\end{theorem}

\begin{proof}
The disk-decomposition and van Kampen argument preceding
\eqref{eq:pi1-polizzi-presentation} shows that inclusion of a regular fiber
surjects on the fundamental group of the total space.  This uses no
section; the possible pointed-clutching relation affects the kernel, not
surjectivity.  By
Theorem~\ref{thm:pi1-polizzi},
\[
\pi_1(S_E)\cong\pi_1(E)\cong\mathbb Z^2,
\]
and the isomorphism is induced by the Albanese morphism
$\alpha:S_E\to E$.  Hence $h$ induces a surjection on $H_1$.  Since a
genus-two fiber meets a general
Albanese fiber in four points, $\deg h=4$ by
\cite[Proposition~6.18]{Polizzi}.

Each $v_i$ bounds its Lefschetz thimble in $S_E$.  Since
$h=\alpha\circ\iota_\Sigma$, where $\iota_\Sigma$ is the fiber inclusion,
it follows that $h_*v_i=0$ and hence $v_i\in K_0$.

Let
\[
h^!:H_1(E;\mathbb Z)\longrightarrow H_1(\Sigma;\mathbb Z)
\]
denote the homological Gysin, or Umkehr, map
\[
h^!=\operatorname{PD}_{\Sigma}^{-1}\circ h^*\circ\operatorname{PD}_E.
\]
The adjunction formula gives
\[
x\cdot h^!y=h_*x\cdot y
\]
and
\[
h^!y_1\cdot h^!y_2
=4(y_1\cdot y_2).
\]
Because $h_*$ is surjective, its Smith normal form has two unit invariant
factors.  The adjoint map $h^!$ is obtained from the transpose of $h_*$ by
multiplication with unimodular symplectic matrices; hence
$h^!H_1(E;\mathbb Z)$ is a primitive rank-two sublattice.  Write
\[
L=h^!H_1(E;\mathbb Z).
\]
Then $L$ has intersection matrix $4J$, where
$J=\left(\begin{smallmatrix}0&1\\-1&0\end{smallmatrix}\right)$, and
\[
K_0=\ker h_*=L^\perp.
\]
The last equality follows directly from the adjunction identity.

By \cref{lem:primitive-complement},
\[
|\det K_0|=|\det L|=16.
\]
A rank-two alternating integral lattice has matrix
$\left(\begin{smallmatrix}0&m\\-m&0\end{smallmatrix}\right)$ in a basis,
so $m^2=16$.  Reversing one basis vector if necessary gives $m=4$.

Choose a basis $p,q$ of $K_0$ with $p\cdot q=4$.  A nonseparating simple
closed curve represents a primitive class in $H_1(\Sigma;\mathbb Z)$.
Since $K_0$ is primitive in $H_1(\Sigma;\mathbb Z)$, each $v_i$ is also
primitive in $K_0$.  Thus, if $v=ap+bq$ is one of the six vanishing classes,
then $\gcd(a,b)=1$, and the Picard--Lefschetz transvection restricted to
$K_0$ has matrix
\[
T_v=
\begin{pmatrix}
1-4ab&4a^2\\
-4b^2&1+4ab
\end{pmatrix}.
\]
Thus $T_v\in\Gamma(4)\subset SL(2,\mathbb Z)$, and because $(a,b)$ is
primitive it is a primitive positive parabolic: it is conjugate in
$SL(2,\mathbb Z)$ to
$\left(\begin{smallmatrix}1&4\\0&1\end{smallmatrix}\right)$.
Indeed, choose $A\in SL(2,\mathbb Z)$ with first column $(a,b)^T$.
Normality of $\Gamma(4)$ gives
$A\left(\begin{smallmatrix}1&4\\0&1\end{smallmatrix}\right)A^{-1}
\in\Gamma(4)$, while the stabilizer in $\Gamma(4)$ of the cusp represented
by $(a,b)$ is the infinite cyclic group generated by this element.  Thus
$T_v$ is primitive in $\Gamma(4)$ itself, not merely in $SL(2,\mathbb Z)$.

Over the regular locus, the maps $\alpha|_{\Sigma_t}$ assemble into a map
of local systems to the constant system $H_1(E;\mathbb Z)$.  Their kernels
therefore form a monodromy-invariant rank-two local subsystem, so the
monodromy preserves $K_0$.  Equivalently, each nonseparating transvection
preserves $K_0$ because its defining vector lies there, while the seven
separating twists act trivially on first homology.  Retain the relative
order of the six nonseparating factors after deleting the seven homologically
trivial factors.  The global genus-two monodromy relation then restricts to
the ordered product in $\operatorname{Aut}(K_0)$:
\[
T_{v_1}T_{v_2}\cdots T_{v_6}=I.
\]
By \cref{lem:gamma-four}, these six transformations are positive primitive
peripheral elements of the six-punctured sphere $Y(4)$.  In its
abelianization the only relation among the oriented peripheral classes is
$e_1+\cdots+e_6=0$.
For a fixed cusp, its stabilizer in $\Gamma(4)$ is infinite cyclic,
generated by the corresponding width-four positive parabolic.  Since each
$T_{v_i}$ is primitive and positive, it is that generator after a choice
of peripheral path.  The positive sign is fixed by the complex orientation
of the Lefschetz singularity and the convention that a right-handed Dehn
twist gives the positively oriented peripheral generator; its class is
therefore $+e_j$, not $-e_j$.
The displayed product contains six positive primitive parabolics and is
the identity.  If the $j$th cusp occurs $n_j$ times, abelianization gives
\[
n_1e_1+\cdots+n_6e_6=0.
\]
Thus $(n_1,\ldots,n_6)$ is a multiple of $(1,\ldots,1)$.  Since
$\sum n_j=6$, every $n_j$ equals one.

The cusps of $\Gamma(4)$ are naturally indexed by
$\mathbb P^1(\mathbb Z/4)$, whose six elements may be represented by
\[
[1:0],\ [0:1],\ [1:1],\ [1:2],\ [1:3],\ [2:1].
\]
This proves the asserted residue classes, up to the natural
$SL(2,\mathbb Z)$ change of the basis $(p,q)$.  In particular no two
$v_i$ are equal up to sign in homology.

It remains to account for the choice of distinguished paths.  One may pass
between distinguished path systems by Hurwitz moves and simultaneous
conjugation.  On $K_0$, every nonseparating Picard--Lefschetz transformation
lies in $\Gamma(4)$ and is therefore the identity modulo $4$, while every
separating twist acts trivially.  More explicitly, a Hurwitz move replaces a
transported class $v_i$ by $T_{v_j}^{\pm1}(v_i)$; since
$T_{v_j}\equiv I\pmod4$ on $K_0$, its cusp label is unchanged.  Thus Hurwitz
moves preserve the cusp labels up to permutation.  A simultaneous
conjugation by $g\in\Mod(\Sigma_2)$ changes the marking and transports the
entire pair
\[
(K_0,\{v_i\})\longmapsto(g_*K_0,\{g_*v_i\}).
\]
The induced lattice isomorphism carries reduction modulo four bijectively
from $K_0/4K_0$ to $g_*K_0/4g_*K_0$.  Thus pairwise distinctness of the
cusp labels is invariant under simultaneous conjugation, although $g_*$
need not preserve the originally marked sublattice $K_0$.  A matching path
can be completed to a distinguished path system and gives two equal
transported unoriented curves, hence two equal cusp labels.  This
contradicts the six distinct labels above.
\end{proof}

\begin{remark}
The last conclusion rules out the most direct matching-sphere approach to
the Luttinger-surgery problem.  Although the six critical values occur in
three natural pairs coming from the three residual bisections, the two
vanishing cycles associated with a single bisection cannot be matching
cycles.  The pairing of ramification values on the branch curve is therefore
different from pairing equal vanishing cycles in the genus-two fiber.
\end{remark}

\section{Elliptic components and the Albanese map}
\label{sec:elliptic-components}

The seven reducible fibers contain additional geometric information that is
useful for the surgery problem.  Polizzi gives the component formulas and
intersection data in \cite[Propositions~6.11, 6.13, and~6.15 and
Corollary~6.14]{Polizzi}.  We combine those explicit curves in $E(3)$ with
the Abel--Jacobi map and \cref{thm:pi1-polizzi}.

Let
\[
\alpha:S_E\longrightarrow E
\]
be the Albanese morphism, obtained by restricting the Abel--Jacobi map
$\beta:E(3)\to E$.  Polizzi writes the seven reducible fibers as
\[
G_1=B_{23}+B_{32},\qquad
G_2=B_{13}+B_{31},\qquad
G_3=B_{12}+B_{21},
\]
and
\[
G_i^*=M_i+N_i,\qquad 1\le i\le4.
\]
All of these components are smooth elliptic curves of self-intersection
$-1$.

\begin{proposition}\label{prop:albanese-on-components}
The restrictions of the Albanese morphism to the elliptic components are
nonconstant finite maps.  After choices of origins they are translations
of isogenies.  More precisely,
\[
\deg(\alpha|_{B_{ij}})=2,\qquad
\deg(\alpha|_{M_i})=1,\qquad
\deg(\alpha|_{N_i})=3.
\]
In particular each restriction induces an injection on fundamental groups whose
image has finite index in $\pi_1(E)$, while
\[
\alpha|_{M_i}:M_i\longrightarrow E
\]
is an isomorphism.  Hence every $M_i$ is an Albanese section, horizontal
for $\alpha$ and vertical as a component of a reducible fiber of the
genus-two fibration, and
\[
\pi_1(M_i)\xrightarrow{\;\cong\;}\pi_1(S_E)\cong\mathbb Z^2.
\]
\end{proposition}

\begin{proof}
The degree statements are exactly the intersection numbers with a general
Albanese fiber recorded by Polizzi in \cite[Propositions~6.11, 6.13,
and~6.15 and Corollary~6.14]{Polizzi}:
\[
F_\alpha\cdot B_{ij}=2,\qquad F_\alpha\cdot M_i=1,\qquad F_\alpha\cdot N_i=3,
\]
where $F_\alpha$ denotes a general Albanese fiber.
They can also be seen from his explicit parametrizations.  For example,
\[
M_i=\{x+\eta_i+2\eta_i\mid x\in E\},
\]
and the Abel--Jacobi sum is $x+3\eta_i=x$, since $\eta_i$ is
$3$-torsion.  Thus $\alpha|_{M_i}$ has degree one and is an isomorphism of
elliptic curves.  More generally, every nonconstant map between elliptic
curves is, after translation, an unramified isogeny by Riemann--Hurwitz and
therefore injects fundamental groups with finite-index image.  Since
Theorem~\ref{thm:pi1-polizzi} identifies
$\alpha_*:\pi_1(S_E)\to\pi_1(E)$ with an isomorphism and
$\alpha_*\circ(\iota_{M_i})_*=(\alpha|_{M_i})_*$, the inclusion
$M_i\hookrightarrow S_E$ also induces an isomorphism on fundamental groups.
For $N_i$, its standard parametrization factors through
$E/\langle\eta_i\rangle$; the induced map to $E$ has degree $3$, in
agreement with $F_\alpha\cdot N_i=3$.
\end{proof}

Fix one of the sections $M=M_i$ and put
\[
C=S_E\setminus\operatorname{int}\nu M.
\]
Since $M^2=-1$, the boundary $\partial\nu M$ is the oriented circle bundle
over $T^2$ with Euler number $-1$.

\begin{proposition}\label{prop:section-complement}
The filling map
\[
\pi_1(C)\longrightarrow\pi_1(S_E)\cong\mathbb Z^2
\]
is surjective and its kernel is normally generated by a meridian $\mu$ of
$M$.  Moreover
\[
H_1(C;\mathbb Z)\cong\mathbb Z^2.
\]
If $x,y$ are lifts to the boundary of standard generators of $\pi_1(M)$,
then on
$\partial\nu M$ one has, up to orientation,
\[
[x,y]=\mu,
\]
and $\mu$ is central in the boundary group.
\end{proposition}

\begin{proof}
The circle bundle $\partial\nu M\to M$ has Euler number $-1$, so its
fundamental group has the standard presentation
\[
\langle x,y,\mu\mid [x,\mu]=[y,\mu]=1,\ [x,y]=\mu^{\pm1}\rangle
\]
after a choice of orientation.  The inclusion
$\pi_1(\partial\nu M)\to\pi_1(\nu M)=\pi_1(M)$ is surjective and its kernel
is normally generated by $\mu$.  Van Kampen for
\[
S_E=C\cup\nu M
\]
therefore identifies the pushout with
\[
\pi_1(S_E)\cong\pi_1(C)/\langle\!\langle i_*(\mu)\rangle\!\rangle.
\]
This proves the asserted statement about the kernel.

For first homology, the Euler number $-1$ implies
\[
H_1(\partial\nu M;\mathbb Z)\cong H_1(M;\mathbb Z)\cong\mathbb Z^2,
\]
because the circle fiber dies in the abelianization.  The map
$H_1(\partial\nu M)\to H_1(\nu M)$ is therefore an isomorphism.  The
Mayer--Vietoris sequence for the same decomposition gives
\[
H_1(C;\mathbb Z)\cong H_1(S_E;\mathbb Z)\cong\mathbb Z^2.
\]
\end{proof}

\begin{remark}\label{rem:section-complement-meaning}
Proposition~\ref{prop:section-complement} shows that the kernel of the
filling map is normally generated by the meridian.  It does not by itself
determine the group of the section complement or the relations among
conjugates of that meridian.  Nevertheless, it suggests working near the
elliptic section $M_i$, where the meridian is represented on the boundary
by the commutator of chosen boundary lifts of the two Albanese generators.
\end{remark}

\section{Necessary geometry for essential Lagrangian tori}
\label{sec:essential-tori}

The preceding section identifies two generators of $\pi_1(S_E)$ on an
elliptic Albanese section.  It does not follow that product tori near that
section are homologically essential.  The following elementary observation
rules out every construction confined to its normal neighborhood.

\begin{lemma}\label{lem:section-neighborhood-null}
Let $M=M_i$ be an elliptic component on which the Albanese map has degree
one.  Every oriented Lagrangian torus contained in $\nu M$ is
nullhomologous in $\nu M$, and hence in $S_E$.
\end{lemma}

\begin{proof}
The disk bundle $\nu M$ retracts onto $M$, and its intersection form is
the one-by-one matrix $[-1]$.  If $T\subset\nu M$ is an oriented
Lagrangian torus, then its normal bundle is isomorphic to its tangent
bundle through an almost-complex structure compatible with the symplectic
form.  Thus $T^2=0$.  Writing $[T]=k[M]$ in $H_2(\nu M;\mathbb Z)$ gives
$0=T^2=-k^2$, so $k=0$.
\end{proof}

There is also a useful restriction on the visible algebraic classes.  Let
$F_\alpha$ be a regular Albanese fiber, let $G$ be a regular genus-two
fiber, and fix an elliptic section $M=M_i$.  Their intersections and
canonical degrees are
\[
 F_\alpha^2=G^2=0,\quad F_\alpha\cdot G=4,\quad
 M^2=-1,\quad F_\alpha\cdot M=1,\quad G\cdot M=0,
\]
and, by adjunction,
\[
 K\cdot F_\alpha=4,\qquad K\cdot G=2,\qquad K\cdot M=1.
\]

\begin{proposition}\label{prop:visible-span-obstruction}
No nonzero class in
$\operatorname{span}_{\mathbb Q}\{[F_\alpha],[G],[M]\}$ can be represented
by a Lagrangian torus.
\end{proposition}

\begin{proof}
For an oriented Lagrangian torus $T$ with respect to the fixed K\"ahler
form on $S_E$, a compatible almost-complex
structure identifies $TS_E|_T$ with the complexification of $TT$; hence
$K\cdot[T]=0$, and also $[T]^2=0$.  Write
$[T]=aF_\alpha+bG+cM$.  The canonical pairing gives
$c=-4a-2b$.  The displayed intersection numbers then give
\[
[T]^2=8ab+2ac-c^2=-4(6a^2+3ab+b^2).
\]
The quadratic form $6a^2+3ab+b^2$ is positive definite, so $[T]^2=0$
forces $a=b=c=0$.
\end{proof}

The abstract intersection lattice does not obstruct the desired pair.  If
$T_1,T_2$ are disjoint essential tori with integral duals $D_1,D_2$, then
necessarily
\[
T_i^2=K\cdot T_i=0,\qquad T_1\cdot T_2=0,
\qquad D_i\cdot T_j=\delta_{ij}.
\]
Thus $\langle[T_1],[T_2]\rangle$ is a primitive rank-two isotropic
sublattice.  Since $\pi_1(S_E)=\mathbb Z^2$, the universal coefficient
theorem and Poincar\'e duality show that $H_2(S_E;\mathbb Z)$ is torsion
free and its intersection pairing is unimodular.  The equalities $e=9$,
$b_1=2$, and $\sigma=-5$ give $b_2=11$ and
$(b_2^+,b_2^-)=(3,8)$.  The class $M$ has $M^2=-1$, so the form is odd;
the classification of indefinite odd unimodular forms therefore gives
$I_{3,8}$.  There is consequently no rank or signature obstruction to an
isotropic pair.  The issue is its geometric realization outside the span
of the natural fibers and sections.

One possible source is the Albanese fibration
$\alpha:S_E\to E$.  The next statement records the exact geometric input
that this approach would require.  It is a criterion, not an existence
assertion.

\begin{proposition}\label{prop:suspension-criterion}
Let $a,b\subset E$ be oriented simple loops meeting once and avoiding the
critical values of $\alpha$.  Suppose the symplectic monodromy along $a$
preserves an oriented embedded curve $c$ in a regular genus-three fiber up
to an orientation-preserving parametrized isotopy, and the monodromy along
$b$ similarly preserves an oriented embedded curve $d$.  Assume that the
closing isotopies are induced by Hamiltonian isotopies of the fiber.  Their parallel-transport
traces are Lagrangian tori $T(a,c)$ and $T(b,d)$, and
\[
[T(a,c)]\cdot[T(b,d)]=(a\cdot b)(c\cdot d).
\]
If the transported curves over the base intersection admit disjoint
representatives compatible with the closing isotopies, the traces can be
separated there.  If, in addition, there are invariant curves $c^*,d^*$ for
the opposite monodromies such that
\[
c\cdot c^*=1,\qquad d\cdot d^*=1
\]
and all cross-pairings vanish geometrically, then $T(b,c^*)$ and
$T(a,d^*)$, with orientations chosen appropriately, give integral duals to
$T(a,c)$ and $T(b,d)$.  In particular, the two
Lagrangian classes are primitive and independent.
\end{proposition}

\begin{proof}
Symplectic parallel transport carries each vertical curve through a
one-parameter family of isotropic curves.  In a collar of the base loop,
attach the Hamiltonian suspension of the closing isotopy; after the usual
exact adjustment, this closes the trace to an embedded Lagrangian torus.
The same local construction may be chosen compatibly for disjoint traces.
Intersections of
two traces occur over intersections of their base loops, and the local
intersection number is the product of the base and fiber intersection
numbers.  This proves the formula.  Applying it to the proposed dual
traces gives the identity intersection matrix, which proves primitivity
and independence.
\end{proof}

The monodromy of the genus-three Albanese fibration along generators of
$\pi_1(E)$ is not determined by the intersection tables used in this
paper.  In particular, the existence of the invariant embedded curves in
Proposition~\ref{prop:suspension-criterion} remains to be established.
Homological invariance of a fiber class would not suffice: the monodromy
must preserve an embedded curve up to parametrized isotopy.
The degree-one elliptic curves $M_i$ provide marked points in the fibers
of $\alpha$ and hence lift the monodromy to a pointed mapping class group,
but this does not produce an essential invariant curve.  The boundary of
a small disk about the marked point becomes nullhomotopic after the
marking is forgotten.  Moreover, the branch calculation in
Section~\ref{sec:F2} concerns the genus-two fibration over $\mathbb P^1$,
not the Albanese fibration over $E$.  Even the action of the latter
monodromy on first homology would be insufficient, since composition with
Torelli elements can change its reduction system without changing that
action.  A curve-level analysis of the Albanese monodromy, with the aim of
constructing the essential Lagrangian tori required by
Proposition~\ref{prop:suspension-criterion}, will be undertaken in forthcoming
work.

\section{Lagrangian tori near an elliptic section}
\label{sec:direct-luttinger}

Fix a section $M=M_i$.  Over a punctured torus
\[
M^\circ=M\setminus D^2
\]
the normal disk bundle is trivial, so
\[
\nu M|_{M^\circ}\cong M^\circ\times D^2.
\]
Choose standard curves $x,y\subset M^\circ$ meeting once and representing
the two generators of $\pi_1(S_E)$.  For two distinct radii
$0<r_1<r_2<1$, the product tori
\[
T_x=x\times C_{r_1},\qquad
T_y=y\times C_{r_2}
\]
are embedded and disjoint: although $x$ and $y$ meet in the base, the two
tori lie at different radii.  By the symplectic neighborhood theorem, and
after choosing horizontal lifts over $M^\circ$, the local coupling form may
be written up to harmless normalization as
\[
\omega=\pi^*\omega_M+d(r^2\vartheta),
\]
and its restriction to either product torus vanishes: the horizontal direction
projects to a one-dimensional curve in $M$, $r$ is constant, and the other
direction is the circle fiber.

These tori are nullhomologous in $S_E$: in the chosen trivialization they
bound the solid tori $x\times D^2_{r_1}$ and $y\times D^2_{r_2}$,
respectively.  They and their surgery neighborhoods are disjoint
from the zero section $M$.  The two bounding solid tori need not be disjoint,
which is one reason the simultaneous complement calculation cannot be
replaced by two independent local calculations.  Fix the Lagrangian
framings.  Choose basing arcs from a common basepoint in the complement to
the two boundary three-tori, and let $x^L\subset\partial\nu T_x$ and
$y^L\subset\partial\nu T_y$ denote the resulting based Lagrangian push-offs.
Fix signs $\varepsilon_x,\varepsilon_y\in\{\pm1\}$.  By unit surgery along
$(T_x,x^L)$, respectively $(T_y,y^L)$, we mean gluing $T^2\times D^2$ so
that its meridian maps to
$\mu_{T_x}(x^L)^{\varepsilon_x}$, respectively
$\mu_{T_y}(y^L)^{\varepsilon_y}$, with the other generators identified by
the Lagrangian framing.  All meridians and push-offs are transported to the
common basepoint in the simultaneous complement along the stated basing
arcs; changing an arc conjugates a word and does not change its normal
closure.  Put
\[
G=\pi_1\bigl(S_E\setminus\operatorname{int}\nu(T_x\cup T_y)\bigr),
\]
and let $Y_{\varepsilon_x,\varepsilon_y}$ denote the result of these two
surgeries.  Van Kampen gives
\[
\pi_1(Y_{\varepsilon_x,\varepsilon_y})\cong
G/\langle\!\langle
\mu_{T_x}(x^L)^{\varepsilon_x},
\mu_{T_y}(y^L)^{\varepsilon_y}
\rangle\!\rangle,
\]
\cite{Luttinger,AurouxDonaldsonKatzarkov}.  Its abelianization already
precludes the desired simply connected construction.
Related applications of Luttinger surgery to the control of fundamental
groups in small symplectic four-manifolds appear in
\cite{AkhmedovZhang}.

Proposition~\ref{prop:section-complement} gives the first global constraint:
the kernel created by removing the elliptic section is normally generated
by its meridian; for chosen boundary lifts one has
\[
i_*(\mu)=[i_*(x),i_*(y)]^{\pm1}.
\]
This information alone does not determine the group obtained after drilling
$T_x\cup T_y$, since the two additional meridians and their conjugates must
also be controlled.

\begin{theorem}\label{thm:nullhomologous-surgery}
Let $Y$ be obtained from $S_E$ by orientation-preserving torus surgeries
with primitive filling slopes along the two
nullhomologous tori $T_x$ and $T_y$.  Then
\[
b_1(Y)\ge2.
\]
In particular, no choice of the two unit Luttinger surgeries above can
produce a simply connected manifold.
\end{theorem}

\begin{proof}
Put
\[
C_T=S_E\setminus\operatorname{int}\nu(T_x\cup T_y).
\]
With rational coefficients, excision and the Thom isomorphism give
\[
H_2(S_E,C_T;\mathbb Q)\cong
H_0(T_x\sqcup T_y;\mathbb Q)\cong\mathbb Q^2,
\qquad
H_1(S_E,C_T;\mathbb Q)=0.
\]
The relevant part of the homology sequence of the pair is therefore
\[
H_2(S_E;\mathbb Q)\longrightarrow\mathbb Q^2
\longrightarrow H_1(C_T;\mathbb Q)
\longrightarrow H_1(S_E;\mathbb Q)\longrightarrow0.
\]
The first arrow records intersection with $[T_x]$ and $[T_y]$.  Both
classes vanish, so this arrow is zero.  Since
$H_1(S_E;\mathbb Q)\cong\mathbb Q^2$, we obtain
\[
0\longrightarrow
\mathbb Q\langle\mu_{T_x},\mu_{T_y}\rangle
\longrightarrow H_1(C_T;\mathbb Q)
\longrightarrow H_1(S_E;\mathbb Q)\longrightarrow0,
\]
and hence $b_1(C_T)=4$.

For the two boundary components, Mayer--Vietoris presents $H_1(Y;\mathbb Q)$
as $H_1(C_T;\mathbb Q)\oplus\mathbb Q^4$ modulo the image of
$H_1(T^3\sqcup T^3;\mathbb Q)$.  On each component two boundary classes
eliminate the two core generators of the filling, while the attaching
meridian contributes at most one additional independent relation in
$H_1(C_T;\mathbb Q)$.  Thus the two fillings impose at most two relations
on the rank-four group $H_1(C_T;\mathbb Q)$, and $b_1(Y)\ge2$.  For the
unit Luttinger fillings these relations are
\[
[\mu_{T_x}]+\varepsilon_x[x^L],\qquad
[\mu_{T_y}]+\varepsilon_y[y^L].
\]
The conclusion is independent of signs, basing arcs, and possible
cross-terms in the push-offs.
\end{proof}

For comparison, suppose that disjoint Lagrangian tori $T_1,T_2$ satisfying
the duality conditions of Proposition~\ref{prop:suspension-criterion} were
found, and put
\[
C=S_E\setminus\operatorname{int}\nu(T_1\cup T_2).
\]
The pair sequence would then identify
$H_1(C;\mathbb Z)\cong H_1(S_E;\mathbb Z)$: the intersection map
$H_2(S_E;\mathbb Z)\to\mathbb Z^2$ is surjective because of the two
duals.  If preferred Lagrangian push-offs mapped to the two Albanese
generators, the two unit surgeries could therefore kill first homology.
This still would not prove simple connectivity.  A proof would require a
presentation of $\pi_1(C)$ with the peripheral words and a verification
that
\[
\pi_1(C)/\langle\!\langle
\mu_1\lambda_1^{\varepsilon_1},
\mu_2\lambda_2^{\varepsilon_2}
\rangle\!\rangle=1.
\]
Dual surfaces make the meridians nullhomologous in the complement; unless
they provide suitably controlled punctured disks, they do not make those
meridians nullhomotopic.
More precisely, if an oriented dual surface $D_i$ of genus $g_i$ meets
$T_i$ once and is disjoint from the other torus, then the punctured surface
$D_i\setminus\operatorname{int}D^2$ gives, up to conjugacy and orientation,
\[
 \mu_i=\prod_{k=1}^{g_i}[u_{ik},v_{ik}]
 \qquad\text{in }\pi_1(C).
\]
Thus a homological dual of positive genus places the meridian in the
commutator subgroup but does not kill it.

\begin{proposition}
\label{prop:dual-sphere-criterion}
Let $T_1,T_2\subset S_E$ be disjoint Lagrangian tori.  Suppose that, for
each $i$, there is an embedded sphere $D_i\subset S_E$ meeting $T_i$
transversely once and disjoint from $T_j$ for $j\ne i$.  Put
\[
C=S_E\setminus\operatorname{int}\nu(T_1\cup T_2).
\]
Then inclusion induces an isomorphism
\[
\pi_1(C)\cong\pi_1(S_E)\cong\mathbb Z^2.
\]
Consequently, if framed Lagrangian push-offs $\lambda_1,\lambda_2$ map to
a basis of $\pi_1(S_E)$, unit Luttinger surgeries along them produce a
simply connected manifold.
\end{proposition}

\begin{proof}
Van Kampen identifies $\pi_1(S_E)$ with the quotient of $\pi_1(C)$ by the
normal closure of the two meridians.  The punctured sphere
$D_i\setminus\operatorname{int}D^2\subset C$ is a disk bounding $\mu_i$,
so both meridians are already trivial in $\pi_1(C)$.  Inclusion is
therefore an isomorphism.  The two surgery relations reduce to
$\lambda_1^{\varepsilon_1}=\lambda_2^{\varepsilon_2}=1$ and kill a basis
of $\mathbb Z^2$.
\end{proof}

The proposition is only a criterion.  No such pair of essential
Lagrangian tori with dual spheres is constructed here.

There is no additional minimality problem in this direct approach.  The
surface $S_E$ is minimal, and Ho--Li prove that Luttinger surgery preserves
symplectic minimality \cite[Proposition~3.1]{HoLi}.  Consequently, if two
unit Luttinger surgeries on homologically essential Lagrangian tori satisfy the displayed
fundamental-group condition, the resulting manifold is simply connected,
minimal, and symplectic, with
\[
e=9,\qquad \sigma=-5,
\]
since Luttinger surgery preserves Euler characteristic and signature.
Freedman's theorem \cite{Freedman} gives the homeomorphism, and the classification of
symplectic structures on rational surfaces \cite{McDuffRational} excludes
a minimal symplectic form on the standard connected sum.  Thus the result
is homeomorphic but not diffeomorphic
to $\mathbb CP^2\#6\overline{\mathbb CP}^{\,2}$.  Thus, for the direct
Polizzi construction, the remaining issue is the existence and complement
geometry of the required tori, not minimality after surgery.

\section{The \texorpdfstring{$d=3$}{d=3} Xiao block and a twisted double}
\label{sec:twisted-43-double}

We finish with a twisted-double application of the seven-fiber Xiao
fibration in the $d=3$ case.  Akhmedov and Monden studied this fibration as
a building block for small exotic four-manifolds \cite{AkhmedovMonden}, and
Akhmedov later computed its ordered branch-cover monodromy
\cite{AkhmedovXiao}.  Using that marking, we specify a gluing class for
which the doubled fibration is simply connected.  The construction is a
monodromy-level variant of the seven-blowup strategy of Akhmedov and
Monden; no diffeomorphism with their particular example is asserted.  We
also briefly recall their four-blowup application for context.

The $d=3$ surface is obtained from the double cover of
$Y=\mathbb CP^2\#7\overline{\mathbb CP}^{\,2}$ branched along the resolved
complete-quadrangle divisor.  The branch divisor consists of the strict
transforms of the six sides together with the strict transform of a nodal
quartic.  Three ramification curves are $(-1)$-sections of the induced
genus-two fibration, while three further ramification curves are vertical
$(-1)$-curves.  Contracting the latter gives Xiao's relatively minimal
model.  The pencil through a triple point cuts the branch divisor in six
points; four simple tangencies and three $3+3$ degenerations give four
nonseparating and three separating vanishing cycles, respectively
\cite{AkhmedovXiao}.

\subsection{The seven-blowup construction}

The twisted-double construction uses this type-$(4,3)$ block twice.
This section is logically independent of the $d=4$ Polizzi analysis above;
no $d=3$ input is used in any of the preceding results.  This block was
previously studied as a Lefschetz-fibration building block by Akhmedov and
Monden \cite{AkhmedovMonden}.  The type records four nonseparating and three
separating vanishing cycles.  Let
\[
f_{\mathrm{cq}}:X_{\mathrm{cq}}\longrightarrow\mathbb P^1
\]
denote this genus-two Lefschetz fibration.  Its total space is the
standard block
\[
X_{\mathrm{cq}}\cong (T^2\times S^2)\#3\overline{\mathbb CP}^{\,2},
\]
and its monodromy has four nonseparating and three separating factors
\cite{AkhmedovXiao,AkhmedovMonden}.  A regular fiber $F\cong\Sigma_2$ has
self-intersection zero, and the three $(-1)$-sections imply that its
meridian is nullhomotopic in $X_{\mathrm{cq}}\setminus\nu F$.

Fix a symplectic basis $(a_1,b_1,a_2,b_2)$ of $H_1(F;\mathbb Z)$.  In the
marking of \cite{AkhmedovXiao}, the four nonseparating vanishing-cycle
classes are
\[
-p,\qquad q,\qquad p-q,\qquad -(p+q),
\]
where
\[
p=b_1+b_2,\qquad q=2a_1+a_2.
\]
They span a primitive rank-two sublattice.  Moreover the quotient by all
seven vanishing cycles is the abelian group $\mathbb Z^2$, and its
abelianization is identified by
\[
Q:H_1(F;\mathbb Z)\longrightarrow\mathbb Z^2,\qquad
Q(x_1,x_2,x_3,x_4)=(x_1-2x_3,\ x_2-x_4).
\]

\begin{lemma}\label{lem:cq-fiber-quotient}
Let $\Pi=\pi_1(F)$ and let $N\triangleleft\Pi$ be the normal closure of
based representatives of the seven vanishing cycles of
$f_{\mathrm{cq}}$.  Then inclusion of a regular fiber induces
\[
 \Pi/N\cong\pi_1(X_{\mathrm{cq}}\setminus\nu F)
 \cong\pi_1(X_{\mathrm{cq}})\cong\mathbb Z^2.
\]
Under abelianization, this quotient is identified with $\mathbb Z^2$ by
$Q$.
\end{lemma}

\begin{proof}
Delete a small disk about a regular value of the fibration.  The remaining
base is a disk containing all seven critical values.  Van Kampen, using
the section to fix basing paths, presents the fundamental group of the
total space over this disk as the fiber group modulo the normal closure of
the seven vanishing cycles.  This total space is
$X_{\mathrm{cq}}\setminus\nu F$.  A punctured $(-1)$-section is a disk
whose boundary is the meridian of $F$; restoring $\nu F$ therefore changes
neither the group nor the quotient presentation.  Finally,
\[
X_{\mathrm{cq}}\cong(T^2\times S^2)\#3\overline{\mathbb CP}^{\,2},
\]
so its fundamental group is $\mathbb Z^2$; see
\cite[Sections~2.5 and~5.1]{AkhmedovMonden}.

The separating vanishing cycles have zero homology class, while the four
nonseparating classes $-p,q,p-q,-(p+q)$ span
$\langle p,q\rangle$.  Hence the abelianization of $\Pi/N$ is
$H_1(F;\mathbb Z)/\langle p,q\rangle$.  Directly from the displayed
coordinates,
\[
 \ker Q=\langle(0,1,0,1),(2,0,1,0)\rangle=\langle p,q\rangle,
\]
and this sublattice is primitive.  Thus $Q$ gives the stated
identification.
\end{proof}

Choose a nonseparating curve $\gamma$ on the second handle with
$[\gamma]=a_2+b_2$; it is disjoint from $a_1$.  Let
\[
\phi=t_{a_1}t_\gamma\in\operatorname{Mod}(\Sigma_2),
\]
with the rightmost twist acting first, and form the twisted fiber sum
\[
Z_\phi=X_{\mathrm{cq}}\#_{F,\phi}X_{\mathrm{cq}}.
\]
Our boundary convention transports the marked fiber of the second block to
the first by the orientation-reversing boundary map
\[
\Psi:F\times S^1\longrightarrow F\times S^1,
\qquad \Psi(x,e^{i\theta})=(\phi(x),e^{-i\theta}).
\]
Thus, if
$\mathcal W=t_{c_1}\cdots t_{c_7}=1$ is the closed-surface factorization of the
block, the doubled factorization is
\[
\mathcal W\mathcal W^\phi=
t_{c_1}\cdots t_{c_7}t_{\phi(c_1)}\cdots t_{\phi(c_7)}=1,
\qquad \mathcal W^\phi=\phi\mathcal W\phi^{-1}.
\]
With the convention $t_v(x)=x-(x\cdot v)v$, the action of $\phi$ on
$H_1(F;\mathbb Z)$ is
\[
\phi_*=\begin{pmatrix}
1&1&0&0\\
0&1&0&0\\
0&0&0&1\\
0&0&-1&2
\end{pmatrix}.
\]
Consequently
\[
\phi_*(p)=(1,1,1,2),\qquad
\phi_*(q)=(2,0,0,-1),
\]
and
\[
Q(\phi_*(p))=(-1,-1),\qquad
Q(\phi_*(q))=(2,1).
\]
These two vectors form an integral basis of $\mathbb Z^2$, since
\[
\det\begin{pmatrix}-1&2\\-1&1\end{pmatrix}=1.
\]

\begin{theorem}\label{thm:twisted-double-simply-connected}
The twisted double $Z_\phi$ is a simply connected minimal symplectic
four-manifold.  Its numerical invariants
are
\[
e(Z_\phi)=10,\qquad \sigma(Z_\phi)=-6,\qquad
\chi_h(Z_\phi)=1,\qquad c_1^2(Z_\phi)=2.
\]
In particular, it is homeomorphic but not diffeomorphic to
$\mathbb CP^2\#7\overline{\mathbb CP}^{\,2}$.
\end{theorem}

\begin{proof}
Let $\Pi=\pi_1(F)$ and let $N\triangleleft\Pi$ be as in
Lemma~\ref{lem:cq-fiber-quotient}.  That lemma supplies the nonabelian
identification $\Pi/N\cong\mathbb Z^2$ and identifies it by $Q$ after
abelianization.  Van Kampen, with the boundary convention above, gives the
amalgamated-product quotient
\[
\pi_1(Z_\phi)
\cong (\Pi/N)*_{\Pi}(\Pi/N)
\cong\Pi/\langle\!\langle N,\phi(N)\rangle\!\rangle.
\]
Changing basing paths only conjugates the generators of the two normal
subgroups.  After quotienting by $N$, the group is already abelian and is
identified with $\mathbb Z^2$ by $Q$.
The second block contains the vanishing cycles
$-\phi(p)$ and $\phi(q)$.  Their images under $Q$ form the unimodular basis
displayed above, so their normal closure kills the remaining group.  Hence
$\pi_1(Z_\phi)=1$.

The fiber-sum formulas give
\[
e(Z_\phi)=2e(X_{\mathrm{cq}})-2e(F)=2\cdot3-2(-2)=10,
\qquad
\sigma(Z_\phi)=2\sigma(X_{\mathrm{cq}})=-6.
\]
The remaining identities follow from
$4\chi_h=e+\sigma$ and $c_1^2=2e+3\sigma$.  Since $Z_\phi$ is simply
connected, its rank-eight intersection form is unimodular.  An even
unimodular form has signature divisible by $8$, whereas this signature is
$-6$.  The form is therefore odd and hence isomorphic to $I_{1,7}$; its
Kirby--Siebenmann invariant vanishes because $Z_\phi$ is smooth.

Choose area-preserving representatives in the gluing class.  Gompf's sum
construction makes $Z_\phi$ symplectic \cite{Gompf}.  Both summands carry
relatively minimal genus-two Lefschetz fibrations; equivalently, neither
pair $(X_{\mathrm{cq}},F)$ contains a symplectic exceptional sphere disjoint
from $F$.  This is the same relative-minimality input used for the Xiao
fiber sums in \cite[Section~5]{AkhmedovMonden}.  The exceptional ruled case
in Usher's theorem would require the genus-two summing surface to be a
section of an $S^2$-bundle, whereas such a section has the genus of the
base, here one.  Hence \cite[Theorem~1.1]{Usher} applies and $Z_\phi$ is
symplectically minimal.  Freedman's
classification \cite{Freedman} gives the asserted homeomorphism.  Finally, the
classification of symplectic structures on rational surfaces rules out a
minimal symplectic form on the standard seven-fold blowup
\cite{McDuffRational}.  A hypothetical diffeomorphism would transport the
minimal symplectic form on $Z_\phi$ to such a form, a contradiction.
\end{proof}

The mapping class $\phi$ also has a standard local surgery interpretation.
Begin with the untwisted sum
$Z_{\mathrm{id}}=X_{\mathrm{cq}}\#_F X_{\mathrm{cq}}$.  In the neck
$F\times S^1$, the disjoint curves $a_1$ and $\gamma$ determine disjoint
Lagrangian rim tori, after the standard local perturbation,
\[
R_\gamma=\gamma\times\mu_F,\qquad
R_{a_1}=a_1\times\mu_F
\]
and composition of the gluing with $t_\gamma$ and $t_{a_1}$ is represented
by unit Luttinger surgeries with slopes
$\gamma\times\{\mathrm{pt}\}$ and $a_1\times\{\mathrm{pt}\}$, respectively
\cite{AurouxDonaldsonKatzarkov}.  Coefficients are measured in the
product, equivalently Lagrangian, framings; their signs are fixed by the
chosen boundary orientation.  Since $a_1$ and $\gamma$ are disjoint, their
twists commute; equivalently, the second slope may be read after transport
through the first surgery.  This is only the local interpretation of the
gluing change.  Simple connectivity follows from the van Kampen calculation
above.

This is a monodromy-level variant of the seven-blowup strategy of Akhmedov
and Monden.  In the present formulation the gluing changes are absorbed into
$\phi=t_{a_1}t_\gamma$.  The branch calculation identifies the relevant
homology classes of the vanishing cycles, and simple connectivity reduces to
the unimodular computation
\[
 \det\bigl(Q(\phi_*(p)),Q(\phi_*(q))\bigr)=1
\]
above.  The resulting manifold has the same homeomorphism type as the
seven-blowup example in \cite{AkhmedovMonden}; no diffeomorphism with the
particular manifold constructed there is asserted.

\subsection{The four-blowup construction}

For context, Akhmedov and Monden also use the $d=3$ Xiao block, fiber-summed
with their torus-surgery block $Y_2(1,m)$, to construct an exotic
$\mathbb CP^2\#4\overline{\mathbb CP}^{\,2}$ and an infinite nonsymplectic
family \cite[Section~5.4]{AkhmedovMonden}.  The ordered $d=3$ factorization
clarifies the source of the Xiao-side vanishing-cycle relations.  We do not
repeat the nonabelian marking and gluing calculation here, and no new
four-blowup theorem is claimed.

We now isolate exactly what is required to pass from seven to six blowups.
There is also an algebro-geometric interpretation of this step when the
$(-4)$-sphere is a holomorphic rational curve in an algebraic model.
Contracting such a curve produces the Wahl singularity
$\frac{1}{4}(1,1)$, and its local $\mathbb Q$-Gorenstein smoothing has the
rational ball as Milnor fiber.  For a general symplectically embedded
$(-4)$-sphere, we use this only as a local topological analogy and do not
assert the existence of an algebraic contraction or smoothing.
Urz\'ua developed the deformation-theoretic framework for such smoothings
and their birational modifications \cite{UrzuaQG}.  More recently, Reyes and
Urz\'ua used Wahl configurations and $\mathbb Q$-Gorenstein smoothings to
construct small exotic symplectic four-manifolds and to formulate geometric
obstructions to this method \cite{ReyesUrzua}.  Our conditional lantern
approach is different in presentation: it seeks the relevant $(-4)$-sphere
inside a Lefschetz-fibration factorization.  Nevertheless, if the lantern
sphere admitted a holomorphic realization in an algebraic model and its
contraction extended to a global $\mathbb Q$-Gorenstein smoothing, then the
proposed six-blowup construction would lie in the same Wahl-smoothing
framework.

\begin{proposition}\label{prop:blowdown-pi1}
Let $S\subset Z_\phi$ be a symplectically embedded sphere with $S^2=-4$,
put $C=Z_\phi\setminus\operatorname{int}\nu S$, and let $\mu$ be the image
in $\pi_1(C)$ of a generator of
$\pi_1(\partial\nu S)=\mathbb Z/4$.  If $Z_\phi(S)$ denotes the rational
blowdown, then
\[
\pi_1(Z_\phi)=\pi_1(C)/\langle\!\langle\mu\rangle\!\rangle,
\qquad
\pi_1(Z_\phi(S))=
\pi_1(C)/\langle\!\langle\mu^2\rangle\!\rangle.
\]
Consequently $Z_\phi(S)$ is simply connected precisely when $\mu^2$
normally generates $\pi_1(C)$.  Since $Z_\phi$ is simply connected, if
$\mu$ is nullhomotopic in $C$, the first quotient formula forces
$\pi_1(C)=1$, and the blowdown is simply connected.  This applies when
$S$ has a geometric dual sphere.
\end{proposition}

\begin{proof}
The first equality is van Kampen for restoring the disk bundle of Euler
number $-4$.  If $B_2$ denotes the rational ball, the inclusion
$L(4,1)=\partial B_2\to B_2$ induces the quotient
$\mathbb Z/4\to\mathbb Z/2$ and sends $\mu$ to its generator.  This
generator is supplied by the boundary and can be eliminated in the van
Kampen pushout, leaving precisely the relation $\mu^2=1$.  This gives the
second equality.  If $S$ has a geometric dual sphere, puncturing the dual
at its intersection with $S$ gives a disk in $C$ bounded by $\mu$.
\end{proof}

\begin{proposition}
\label{prop:conditional-six-blowup}
Suppose the doubled factorization of $Z_\phi$ is Hurwitz equivalent to one
containing, as a contiguous subword, the boundary multitwist of an embedded
four-holed sphere, and suppose that the corresponding configuration
determines a symplectically embedded $(-4)$-sphere $S\subset Z_\phi$.  Assume
that replacing the four boundary twists by the three interior twists in the
lantern relation realizes the rational blowdown of $S$
\cite{EndoGurtas}, and that $S$ satisfies the
normal-generation condition of Proposition~\ref{prop:blowdown-pi1}.
Then $Z_\phi(S)$ is homeomorphic but not diffeomorphic to
$\mathbb CP^2\#6\overline{\mathbb CP}^{\,2}$.
\end{proposition}

\begin{proof}
Rational blowdown along the sphere $S$ changes the
invariants by
\[
(e,\sigma,c_1^2)=(10,-6,2)\longmapsto(9,-5,3).
\]
Proposition~\ref{prop:blowdown-pi1} gives simple connectivity.  The
intersection form is unimodular of rank seven and signature $-5$; it is odd
because an even unimodular form has signature divisible by eight.  Hence it
is $I_{1,6}$, and its Kirby--Siebenmann invariant vanishes because the
manifold is smooth.  Freedman's theorem \cite{Freedman} gives an orientation-preserving
homeomorphism.  The rational
blowdown is the symplectic sum of $Z_\phi$ along $S$ with $\mathbb CP^2$
along a conic.  Since $Z_\phi$ is minimal, it contains no exceptional
sphere and, in particular, no pair of disjoint exceptional spheres each
meeting $S$ once.  Thus the exceptional-sphere alternatives in
Dorfmeister's rational-blowdown theorem do not occur, and that theorem shows
that $Z_\phi(S)$ is minimal \cite[Theorem~1.2]{Dorfmeister}.  The
classification of symplectic structures on rational surfaces
\cite{McDuffRational} excludes a minimal symplectic form on the standard
connected sum.  A hypothetical diffeomorphism would transport the form on
$Z_\phi(S)$ to such a form, a contradiction.
\end{proof}

The marked homological data used above do not exhibit a lantern subword.
In a closed genus-two surface, an essential four-holed sphere whose four
boundary components are essential has Euler characteristic zero complement,
which is the disjoint union of two annuli.  Its four boundary curves occur
in two parallel pairs, so the boundary multitwist has the form
$t_a^2t_b^2$ for two disjoint
nonseparating curves $a,b$.  For the displayed doubled factorization, the
four nonseparating classes in the first block are
\[
-p,\qquad q,\qquad p-q,\qquad -(p+q),
\]
and their pairwise algebraic intersections are nonzero.  Moreover, none of
their four $\phi$-images agrees up to sign with a first-block class.
Consequently commuting or merely reordering the displayed factors cannot
produce the required repeated boundary twists.  A genuinely nontrivial
sequence of Hurwitz moves would be necessary, and the curve-level data
needed to verify such a sequence are not determined here.
Nor do the paired $(-1)$-sections provide the missing configuration: after
the fiber sum they give $(-2)$-spheres, not a $(-4)$-sphere.  Finding the
lantern sphere and verifying its meridian condition are the remaining
geometric steps in the six-blowup construction.

\section{Concluding remarks}

The local product tori near an elliptic section cannot solve the
six-blowup problem: their nullhomology prevents the two surgeries from
eliminating the rank-two first homology.  A direct construction from $S_E$
would require homologically essential Lagrangian tori whose preferred
surgery curves map onto the two Albanese generators, together with control
of their simultaneous complement.  Proposition
\ref{prop:suspension-criterion} gives a concrete formulation of this route.

The independent twisted-double construction proves the exotic
seven-blowup case.  Propositions~\ref{prop:blowdown-pi1} and
\ref{prop:conditional-six-blowup} record the precise, still conditional,
lantern criterion for passing to six blowups; no six-blowup construction is
claimed here.  The essential-torus and lantern approaches developed above
provide two concrete routes to the six-blowup problem, whose remaining
geometric aspects will be pursued in forthcoming work.

\section*{Acknowledgments}

This work was inspired by the foundational contributions of Gang Xiao and
Francesco Polizzi to the study of these fibrations and their
algebro-geometric models.  The authors are grateful to Francesco Polizzi;
his explicit treatment of the branch model provided an important starting
point for the present paper.  Akhmedov also thanks him for generous
discussions in 2020 and 2021.  The authors also thank Sai-Kee Yeung and
Giancarlo Urz\'ua, from whose work they have learned much about the
algebro-geometric aspects of these subjects.  The authors also thank
\c{C}a\u{g}r\i{} Karakurt, Selman Akbulut, Josef Dorfmeister, and Weiyi
Zhang, from whose work they have learned much about the topological and
symplectic-topological aspects of these subjects.

An LLM-based tool was used for limited assistance with grammar checking and
figure preparation.  All mathematical ideas, arguments, and results are due
to the authors.  The authors checked all mathematical computations and take full responsibility for the contents of
the paper.

\end{document}